\documentclass[a4paper, 12pt]{amsart}
\usepackage{amsmath,amssymb,amsthm}
\usepackage{amscd}
\usepackage{array,enumerate}
\newcommand{\ip}[1]{\mathopen{\langle}#1\mathclose{\rangle}}
\newtheorem{Thm}{Theorem}[section]
\newtheorem{Prop}[Thm]{Proposition}
\newtheorem{Lem}[Thm]{Lemma}
\newtheorem*{claim}{Claim}
\newtheorem{Cor}[Thm]{Corollary}

\theoremstyle{definition}
\newtheorem{Rem}[Thm]{Remark}

\newtheorem{Def}[Thm]{Definition}

\newcommand{\Cs}{C$^\ast$}
\newcommand{\Ws}{W$^\ast$}
\newcommand{\sd}{^{\ast\ast}}
\newcommand{\id}{\mbox{\rm id}}

\newcommand{\rc}{\mathop{\rtimes _{\mathrm r}}}

\newcommand{\rca}[1]{\mathop{\rtimes _{{\mathrm r}, #1}}}

\newcommand{\KK}{\mathrm{KK}}
\newcommand{\IB}{\mathbb B}
\newcommand{\IC}{\mathbb C}
\newcommand{\IK}{\mathbb K}
\newcommand{\IM}{\mathbb M}
\newcommand{\IN}{\mathbb N}

\newcommand{\IT}{\mathbb T}
\newcommand{\IZ}{\mathbb Z}

\newcommand{\cH}{\mathcal H}
\newcommand{\cK}{\mathcal K}

\newcommand{\cE}{\mathcal E}
\newcommand{\cZ}{\mathcal Z}
\newcommand{\cM}{\mathcal M}
\newcommand{\ve}{\varepsilon}
\newcommand{\vp}{\varphi}

\newcommand{\btimes}{\mathbin{\bar{\otimes}}}

\newcommand{\rQ}{\mathrm{Q}}

\newcommand{\acts}{\curvearrowright}

\DeclareMathOperator{\Ad}{Ad}

\DeclareMathOperator{\spa}{\mathop{\overline{span}}}

\DeclareMathOperator{\supp}{supp}

\DeclareMathOperator{\SL}{SL}

\DeclareMathOperator{\SO}{SO}

\DeclareMathOperator{\bigfp}{\lower0.25ex\hbox{\LARGE $\ast$}}
\newcommand{\ad}{\mathop{\rm Ad}}
\title[Amenable \Cs-dynamical systems]
{On characterizations of amenable \Cs-dynamical systems and new examples}
\author{Narutaka Ozawa}
\author{Yuhei Suzuki}
\address{RIMS, Kyoto University, 606-8502 Kyoto, Japan}
\email{narutaka@kurims.kyoto-u.ac.jp}
\address{Department of Mathematics, Faculty of Science, Hokkaido University,
Kita 10, Nishi 8, Kita-Ku, Sapporo, Hokkaido, 060-0810, Japan}

\email{yuhei@math.sci.hokudai.ac.jp}
\begin{document}
\maketitle
\begin{abstract}
We show the equivalence of several amenability type conditions for \Cs-dynamical systems.
As an application, based on the Pimsner--Meyer construction,
we give a new method to produce
amenable actions on simple \Cs-algebras.
\end{abstract}

\tableofcontents
\section{Introduction}

The notion of amenability for non-singular actions of locally compact groups 
on measure spaces was introduced by Zimmer \cite{Zim}. Later the notion was extended 
for actions on von Neumann algebras by Anantharaman-Delaroche \cite{AD79}. 
The study of the topological counterpart, namely, the notion of amenability for 
\Cs-dynamical systems, was initiated by Anantharaman-Delaroche \cite{AD}.
This has seen a great success, particularly in the case of actions 
on commutative \Cs-alge\-bras, in connection with the study of 
exactness \cite{HR}, \cite{KW}, \cite{Oz}.
For some significant applications of amenable actions, see \cite{Hig}, \cite{Tu}, and Chapter 15 of \cite{BO} for instance.

Recently there is a resurgence of interest in the study of amenability 
for group actions on \emph{non\-commu\-ta\-tive} \Cs-algebras \cite{BC}, \cite{BEW}, \cite{BEW2}, \cite{EN}, \cite{Suz20c}.
This is due to emergence of new interesting examples and phenomena of amenable \Cs-dynamical systems \cite{Suzeq}, \cite{Suz20c}.
After the seminal work \cite{AD}, a number of definitions have been proposed for amenability type conditions for 
\Cs-dynamical systems (see e.g., \cite{AD02}, \cite{Ex}, \cite{EN}, \cite{BEW}, \cite{BEW2}, \cite{BC}).
Recently many of them have been shown to be equivalent in some cases, 
by Buss--Echterhoff--Willett \cite{BEW2},
and by Bearden--Crann \cite{BC}. 
For details, we refer the reader to \cite{BEW2}, \cite{BC}, and references therein.

The purpose of this article is two fold. 
First, we prove further equivalence among amenability type conditions 
for \Cs-dynamical systems. 
Notably, we establish the equivalence of the quasi central approximation property \cite{BEW},
the Exel--Ng approximation property \cite{EN}, and (von Neumann) amenability \cite{BEW2}, \cite{BC} in full generality.
This closes the circle of implications that were left open (see Section 8.1 of \cite{BEW2} for details).
We note that this is already proved in \cite{BEW2} for \emph{discrete} and \emph{unital} \Cs-dynamical systems. 
However, we emphasize that even if one is interested only in discrete and unital \Cs-dyna\-mical 
systems, there is a need to develop the theory for 
locally compact group actions on non-unital \Cs-alge\-bras. 
The need originates from our second main result, a construction of 
new examples of amenable \Cs-dynamical systems (Section~\ref{Sec:Example}).
Unlike the previously known ad hoc constructions \cite{Suzeq}, \cite{Suz20c},
our new construction is \emph{functorial}. More precisely,
we show that the equivariant Pimsner construction \cite{Pim97}, recently studied by Meyer \cite{Mey},
preserves amenability.
This provides plenty of new interesting examples of amenable actions on \emph{simple} \Cs-algebras.
The proof of the amenability of the Pimsner--Meyer system
$G\acts B$ is built on the analysis of the fixed point algebra $B^{\IT}$ of the gauge action by the circle group $\IT$.
This involves the study of associated \Cs-dyna\-mical systems 
$G\times\IT\acts B$, $G\acts B\rtimes\IT$, and their subsystems, which are no longer discrete and unital.
(For details, see Section \ref{Sec:fixed point}.)

\section{Preliminaries}

We fix notations and collect some results from the literature.

\subsection{Locally compact groups}
Throughout the article, the symbol $G$ stands for a locally compact (Hausdorff) group.
For basic facts on locally compact groups, we refer the reader to the book \cite{Fol}.
Unless otherwise specified, we do not assume any conditions on $G$.
Denote by $e\in G$ the unit element.
Denote by $m$ the Haar measure on $G$.
Denote by $\Delta \colon G \rightarrow \mathbb{R}_{>0}$ the modular function of $G$.
Put \[P_c(G) := \left\{ f\in C_c(G): f \geq 0, \int_G f\,dm =1 \right\}.\]

\subsection{\Cs- and \Ws-dynamical systems}
A $G$-\Cs-algebra is a pair $(A, \alpha)$
of a \Cs-algebra $A$ and an action $\alpha \colon G \curvearrowright A$
which is point-norm continuous.
A $G$-\Ws-algebra is a pair $(M, \alpha)$
of a von Neumann algebra $M$ and an action $\alpha \colon G \curvearrowright M$
whose induced action $G \curvearrowright M_\ast$ is point-norm continuous.
Such pairs are also referred to as a \Cs-dynamical system, a \Ws-dynamical system, respectively.
A map between \Cs- or \Ws-dynamical systems is called a $G$-map when it is $G$-equivariant. For instance, a $G$-equivariant unital completely positive (u.c.p.) map
is called a $G$-u.c.p.~ map.

For a $G$-\Ws-algebra $(M, \alpha)$,
we equip $L^\infty(G) \btimes M$ with
the diagonal action $\tilde{\alpha}$ of the left translation action and $\alpha$.
We identify $M$ with the $G$-\Ws-subalgebra $\IC 1_{L^\infty(G)} \otimes M$ of $L^\infty(G) \btimes M$ in the obvious way.

Let $(A, \alpha)$ be a $G$-\Cs-algebra.
A function $k \colon G \rightarrow A$ is said to be
of \emph{positive type}
if for any finite subset $F \subset G$,
the matrix
\[[\alpha_s(k(s^{-1} t))]_{s, t\in F} \in \IM_F(A)\]
is positive.

\subsection{The $G$-\Cs-correspondence $L^2(G, A)$}
For a $G$-\Cs-algebra $(A, \alpha)$,
the associated $G$-\Cs-correspondence $L^2(G, A)$ over $A$ is fundamental
to formulate amenable actions. We first recall the definition.
Let $C_c(G, A)$ denote the set of all compactly supported continuous $A$-valued functions on $G$.
We equip $C_c(G, A)$ with the $A$-bimodule structure given by
\[(a \xi)(g):= a\cdot \xi(g), \quad (\xi a)(g):= \xi(g) a \quad{\rm~for~}a\in A,~ \xi\in C_c(G, A),~ g\in G.\]
On $C_c(G, A)$, define the $A$-valued inner product $\langle~, ~\rangle \colon C_c(G, A) \times C_c(G, A) \rightarrow A$
by
\[\langle \xi, \eta \rangle:= \int_G \xi(g)^\ast \eta(g) \,dm(g).\]
Define $\tilde{\alpha} \colon G \curvearrowright C_c(G, A)$ by 
\[[\tilde{\alpha}_g(\xi)](h):= \alpha_g(\xi(g^{-1}h))\quad {\rm~for~}\xi\in C_c(G, A),~g, h\in G.\]
Now define $L^2(G, A)$ to be the completion of $C_c(G, A)$ with respect to $\langle~, ~\rangle$.
Thus $L^2(G, A)$ is a Banach space with respect to the norm
\[\|\xi\| := \|\langle \xi, \xi \rangle\|^{1/2} \quad{\rm~for~}\xi \in L^2(G, A).\]
The $A$-bimodule structure on $C_c(G, A)$ and $\tilde{\alpha}$
continuously extend to $L^2(G, A)$.
We denote the extended operations and action by the same symbols.
These operations make $L^2(G, A)$ a $G$-\Cs-correspondence over $A$.
That is, $L^2(G, A)$ is a \Cs-correspondence over $A$ satisfying
\[\tilde{\alpha}_g(a \xi)= \alpha_g(a)\tilde{\alpha}_g(\xi),\quad\tilde{\alpha}_g(\xi a)= \tilde{\alpha}_g(\xi)\alpha_g(a),\quad \ip{\tilde{\alpha}_g(\xi), \tilde{\alpha}_g(\eta)} = \alpha_g(\ip{\xi, \eta})\]
for $g\in G$, $a\in A$, $\xi, \eta \in L^2(G, A)$.
Note that $\tilde{\alpha}$ extends to the left $L^1(G)$-action
\[\tilde{\alpha}_f(\xi):=\int_G f(g)\tilde{\alpha}_g(\xi) \,dm(g),\quad f\in L^1(G),~ \xi\in L^2(G, A).\]
When the underlying action $\alpha$ is clear from the context, we also write
\[g \ast \xi := \tilde{\alpha}_g(\xi), \quad f \ast \xi:= \tilde{\alpha}_f (\xi)\quad {\rm~for~} \xi \in L^2(G, A),~ g \in G,~ f\in L^1(G).\]
These notations are compatible with the convolution products on $L^1(G)$.
For $a\in A$ and $\xi \in L^2(G, A)$, put
\[[\xi, a] :=\xi a- a\xi.\]
\subsection{Amenability of \Ws-dynamical systems}
Recently Bearden--Crann \cite{BC} established the following useful 
characterizations of amenability of $G$-\Ws-algebras, 
by extending the work of Anantharaman-Delaroche \cite{AD} in the case of discrete groups.
Note that for exact groups, the statement is shown in \cite{BEW2} by a different method.
\begin{Thm}[\cite{BC}, Theorem 1.1]\label{Thm:BC}
Let $\alpha\colon G \curvearrowright M$ be a \Ws-dynamical system. 
The following conditions are equivalent:
\begin{enumerate}[\upshape(1)]
\item The action $\alpha$ is amenable; that is, 
there is a $G$-conditional expectation
\[L^\infty(G) \btimes M \rightarrow M.\]
\item The action $\alpha$ is amenable when restricted to the center $\cZ(M)$ of $M$; that is, 
there is a $G$-conditional expectation
\[L^\infty(G) \btimes \cZ(M) \rightarrow \cZ(M).\]
\item There is a net $(\xi_i)_{i \in I}$ in $C_c(G, \cZ(M))$
satisfying $\langle \xi_i, \xi_i \rangle = 1$ for $i\in I$
and that $\langle \xi_i, \tilde{\alpha}_g(\xi_i) \rangle \rightarrow 1$ uniformly on compact subsets of $G$ in the ultraweak topology. 
\end{enumerate}
\end{Thm}

\subsection{Exact groups and amenable actions}
Amenability of group actions is first introduced by Zimmer \cite{Zim} in the measurable context.
This is a powerful concept to capture amenable sides of non-amenable groups.
Later its continuous analogue is introduced by Anantharaman-Delaroche.
Here we recall the definition.

Let $\mathrm{Prob}(G) \subset C_0(G)^\ast$ denote the space of regular probability measures on $G$ equipped with the weak-$\ast$ topology.
We equip $\mathrm{Prob}(G)$ with the left translation $G$-action.
\begin{Def}
Let $\alpha \colon G \curvearrowright X$
be a continuous action of $G$ on a locally compact (Hausdorff) space.
The action $\alpha$ is said to be $($topologically$)$ \emph{amenable}
if there is a net
\[m_i \colon X \rightarrow \mathrm{Prob}(G);~i\in I\]
of continuous maps satisfying
$\|m_i(\alpha_g(x)) - (g.m_i)(x)\| \rightarrow 0$ uniformly on compact subsets of $G \times X$.
\end{Def}
We identify the action $G \curvearrowright X$
with the associated action $G \curvearrowright C_0(X)$ via the Gelfand duality.
In particular we use the same symbol for these two actions.

Anantharaman-Delaroche establishes useful characterizations of amenable dynamical systems \cite{AD}, \cite{AD02}.
\begin{Thm}[\cite{AD02}, Proposition 2.5]
Let $\alpha \colon G \curvearrowright X$ be a continuous action on a locally compact space $X$.
The following conditions are equivalent.
\begin{enumerate}[\upshape(1)]
\item The action $\alpha$ is amenable.
\item There is a net $(k_i)_{i\in I} \subset C_c(G, C_0(X))$ of positive type functions
with $k_i(e)\leq 1$
which strictly converges to $1$ uniformly on compact subsets of $G$.
\item There is a net $(\xi_i)_{i\in I}$ in $L^2(G, C_0(X))$ satisfying
$\| \xi_i\| \leq 1$ for $i\in I$
and that
$\langle \xi_i, \tilde{\alpha}_g(\xi_i)\rangle \rightarrow 1$
uniformly on compact subsets of $G$ in the strict topology.
\end{enumerate}
\end{Thm}

When a group admits an amenable action
on a $($non-empty$)$ \emph{compact} space, there are some consequences in \Cs-algebra theory.
Motivated by such phenomena, Anantharaman-Delaroche
has introduced the following definition.
\begin{Def}[\cite{AD}, Definition 3.1]
A locally compact group $G$ is said to be \emph{amenable at infinity}
if it admits an amenable action on a compact space.
\end{Def}
Denote by $C^{\rm lu}_{\rm b}(G)$
the \Cs-algebra of all bounded left uniformly continuous $\mathbb{C}$-valued functions on $G$.
It is not hard to see that for any continuous action $G \curvearrowright X$ on a compact space $X$,
there is a unital $G$-$\ast$-homomorphism
$C(X) \rightarrow C^{\rm lu}_{\rm b}(G)$.
Therefore, when $G$ is amenable at infinity,
the left translation action $G \curvearrowright C^{\rm lu}_{\rm b}(G)$ is amenable.

Amenability at infinity has a strong connection with \emph{exactness} of groups.
Recall that a locally compact group $G$ is said to be exact \cite{KW}
if the reduced crossed product functor $- \rc G$ preserves
short exact sequences.
Ozawa \cite{Oz} has shown that 
for discrete groups, amenability at infinity is equivalent to exactness.
For locally compact second countable groups, in \cite{BCL},
based on metric geometry (see e.g., \cite{DL}, \cite{RW}, \cite{Sak}),
the same equivalence is established.
Here we extend the statement to general locally compact groups $G$.

\begin{Prop}\label{Prop:exact}
A locally compact group $G$ is exact if and only if it is amenable at infinity.
\end{Prop}
\begin{proof}
By Theorem 7.2 in \cite{AD02}, amenability at infinity implies exactness.
We have to show the converse. 
Assume $G$ is exact.

We first consider the case that $G$ is compactly generated.
By the Kakutani--Kodaira theorem \cite{KK}, one has a compact normal subgroup
$K \lhd G$ with the second countable quotient group $G/K$.
Since $G$ is exact, so is $G/K$ \cite{KW}.
By Theorem A of \cite{BCL}, one has an amenable action
$G/K \curvearrowright X$ on a compact space.
Since $K$ is compact,
the inherited action
$G \curvearrowright X$ is amenable.

In general, since $G$ is locally compact,
it is the directed union
of compactly generated open subgroups $(G_i)_{i \in I}.$
Note that exactness passes to open subgroups (\cite{KW}, Corollary 3.5). Hence
each action $G_i \curvearrowright C^{\rm lu}_{\rm b}(G_i)$
is amenable.
Observe that, for each $i\in I$, the right $G_i$-coset decomposition of $G$
gives a unital $G_i$-embedding
$C^{\rm lu}_{\rm b}(G_i) \rightarrow C^{\rm lu}_{\rm b}(G)$.
This yields the amenability of $G \curvearrowright C^{\rm lu}_{\rm b}(G)$.
\end{proof}

\subsection{Universal enveloping $G$-\Ws-algebras}\label{Subsec:Wst}
For a $G$-\Cs-algebra $(A, \alpha)$,
unfortunately the extended action $\alpha\colon G\acts A\sd$ may not 
be continuous. 
Here we recall Ikunishi's universal enveloping $G$-\Ws-algebra
of a $G$-\Cs-algebra \cite{Ik}. This is a replacement of the second dual $A\sd$
in the equivariant setting. Our presentation follows that of \cite{BEW2}.
A somewhat similar work is found in \cite{BBHP}.

Consider the inclusions \[A\subset \cM(A\rtimes G)\subset (A\rtimes G)\sd.\]
Define $A''_\alpha$ to be the ultraweak closure of $A$ in $(A\rtimes G)^{**}$.
By the universality of the full crossed product, this definition coincides with
Ikunishi's original definition \cite{Ik}.
(In fact, one can also replace the full crossed product by the reduced one.
This follows from the proof of the lemma below.)

For a \Cs-algebra $A$, denote by $\mathfrak{S}(A)$ the state space of $A$,
equipped with the weak-$\ast$ topology.
\begin{Def}
Let $(A, \alpha)$ be a $G$-\Cs-algebra.
We say that a state $\varphi \in \mathfrak{S}(A)$ is \emph{$G$-continuous}
if the map $G \ni g \mapsto \varphi \circ \alpha_g^{-1}$
is continuous in norm.
\end{Def}

We record the following useful characterization of the predual of $A''_\alpha$.
\begin{Lem}\label{Lem:contist}
A state $\vp \in \mathfrak{S}(A)$ belongs to $(A''_\alpha)_\ast \subset A^\ast$ if and only if 
it is $G$-continuous.
Thus the extended action $\alpha\colon G\acts A''_\alpha$ 
is continuous in the \Ws-sense.
\end{Lem}
\begin{proof}
Assume $\varphi \in \mathfrak{S}(A)$ is $G$-continuous.
Then for any $\ve>0$ there is $f\in P_c(G)$ 
such that $\psi:=\int_G f(g) (\vp \circ\alpha_g^{-1})\,dm(g)$ 
satisfies $\|\psi-\vp\|<\ve$. 
Let $(\pi, \cH, \xi)$ be the GNS-triple of $\varphi$.
Consider the $\ast$-representation of $A\rtimes G$ on $L^2(G,\cH)$ 
given by the covariant representation $(\tilde{\pi}, \lambda \otimes \id_{\cH})$,
where $\tilde{\pi}\colon A\to \IB(L^2(G,\cH))$ is given by
\[(\tilde{\pi}(a)\xi)(g):=\pi(\alpha_g^{-1}(a))\xi(g),\quad a\in A, \xi \in L^2(G, \cH), g\in G.\] 
Then one has $\psi(a)=\ip{f^{1/2}\otimes\xi, \tilde{\pi}(a) (f^{1/2}\otimes\xi)}$ for $a\in A$. Hence
$\psi\in (A''_\alpha)_\ast$. 
Since $\ve>0$ was arbitrary, this implies $\vp\in (A''_\alpha)_\ast$.

Conversely, assume that $\vp \in \mathfrak{S}(A)$ defines a normal state on $A''_\alpha$. 
By definition, the state $\vp$ comes from a state $\psi$ on $A \rtimes G$. 
Let $(\pi, \cH, \xi)$ be the GNS-triple of $\psi$.
Then, in turn, $\pi$ gives rise to a covariant representation $\pi_A\colon A\to\IB(\cH)$ and 
$u\colon G \curvearrowright \cH$. Note that the unitary representation $u$ is strongly continuous. 
Now $G$-continuity of $\vp$ follows from the formula 
$(\vp\circ\alpha_g^{-1})(a)=\ip{u_g\xi,\pi_A(a)u_g\xi}$, $a\in A$. 
\end{proof}
For a contractive completely positive (c.c.p.)~ map $\Phi \colon A \rightarrow B$ between \Cs-algebras,
$\Phi$ is said to be \emph{non-degenerate} if
$\Phi((A)_1)(B)_1$ is dense in $(B)_1$.
Here $(X)_1$ denotes the closed unit ball of a Banach space $X$. 
\begin{Lem}\label{Lem:extension}
Let $(A, \alpha)$ and $(B, \beta)$ be $G$-\Cs-algebras.
Let $\Phi \colon A \rightarrow B$ be a non-degenerate $G$-c.c.p.~ map.
Then $\Phi$ admits a normal $G$-u.c.p.~ extension
$\Phi'' \colon A''_\alpha \rightarrow B''_\beta$. 
\end{Lem}
\begin{proof}
Consider the dual map $\Phi^\ast \colon B^\ast \rightarrow A^\ast$.
Since $\Phi^\ast$ is $G$-equivariant, by Lemma \ref{Lem:contist},
$\Phi^\ast$ restricts to the map $\Phi^\ast_{\rm c} \colon (B''_{\beta})_\ast \rightarrow (A''_{\alpha})_\ast$.
The dual map $\Phi'':=(\Phi^\ast_{\rm c})^\ast$
gives a normal extension of $\Phi$. Since $\Phi$ is non-degenerate, $\Phi''$ is unital.
\end{proof}

For any Hilbert space $\cH$ and a von Neumann algebra $M\subset\IB(\cK)$, 
denote by $\cH\btimes M$ the Hilbert \Ws-module 
\begin{align*}
\cH\btimes M &:= \{ T \in \IB(\cK,\cH\otimes\cK) : T x'=(\id_{\cH}\otimes x')T\mbox{ for all }x'\in M'\}\\
&\cong \{ T(P_0\otimes 1) : T \in \IB(\cH)\btimes M\},
\end{align*}
where $P_0$ is any fixed rank-one projection on $\cH$. 
Note that in the second presentation,
the $M$-valued inner product is given by
$\ip{S,T}:=S^\ast T\in P_0 \btimes M =M$ for $S, T\in \cH\btimes M$. 
If $\cH_0\subset\cH$ is a dense subspace and $M_0\subset M$ is an ultraweakly 
dense $\ast$-subalgebra, then the unit ball of the algebraic tensor product
$\cH_0 \odot M_0$ is ultrastrongly dense in 
$(\cH\btimes M)_1$ by Kaplansky's density theorem (applied to $\IB(\cH)\btimes M$). 
In particular, when $A \subset M$ is an ultrastrongly dense $\ast$-subalgebra,
the unit ball of $C_c(G, A)$ is ultrastrongly dense in the unit ball of $C_c(G, M) \subset L^2(G) \btimes M$.

\subsection{Amenability conditions for \Cs-dynamical systems}
We review amenability type conditions for \Cs-dynamical systems.

The first definition is based on enveloping \Ws-algebras.
\begin{Def}[\cite{AD}, \cite{BEW2}, \cite{BC}]
A \Cs-dynamical system $\alpha\colon G \curvearrowright A$ is said to be \emph{amenable}
if the universal enveloping \Ws-dynamical system $\alpha\colon G \curvearrowright A''_\alpha$
is amenable (in the sense of Theorem~\ref{Thm:BC}).
\end{Def}
The next two definitions are based on the $G$-\Cs-correspondence $L^2(G, A)$.
\begin{Def}[\cite{EN}, Definition 3.6, see also Definition 3.27 in \cite{BEW2}]
A $G$-\Cs-algebra $(A, \alpha)$ is said to have
the \emph{Exel--Ng approximation property}
if
there is a bounded net $(\xi_i)_{i\in I}$ in $L^2(G, A)$
with
\[\langle \xi_i, a \tilde{\alpha}_g(\xi_i)\rangle \rightarrow a\]
in norm uniformly on compact subsets of $G$ for each $a\in A$.
\end{Def}
Recently, motivated by examples in \cite{Suzeq}, a different amenability type condition is introduced in \cite{BEW}.
\begin{Def}[\cite{BEW}, \cite{BEW2}]
A $G$-\Cs-algebra $(A, \alpha)$ is said to have
the \emph{quasi-central approximation property} (QAP)
if there is a net $(\xi_i)_{i\in I}$ in $L^2(G, A)$ satisfying
\[\| \xi_i\| \leq 1 {\rm~for~all~}i\in I,\quad \|[\xi_i, a]\| \rightarrow 0 {\rm~for~all~} a\in A,\]
and
\[\langle \xi_i, \tilde{\alpha}_g(\xi_i)\rangle \rightarrow 1\]
uniformly on compact subsets of $G$ in the strict topology.

Note that the third condition can be replaced by the following condition:
\noindent $\ip{\xi_i, \xi_i} \rightarrow 1$ strictly, and for every $a\in A$,
$\|(\tilde{\alpha}_g (\xi_i) - \xi_i) a\| \rightarrow 0$ uniformly on compact subsets of $G$.\end{Def}

We will show that these two conditions are in fact equivalent.
We first recall the following lemma.
This is a part of Kasparov's technical lemma (Lemma in page 152 of \cite{Kas}).
As the part we use is easy to show,
for reader's convenience, we include a proof.
\begin{Lem}\label{Lem:qa}
Let $(A, \alpha)$ be a $G$-\Cs-algebra.
Then there is a \emph{$G$-approximate unit} of $(A, \alpha)$;
that is, there is an approximate unit $(e_i)_{i\in I}$ of $A$ with
$\| \alpha_g(e_i) - e_i\| \rightarrow 0$ uniformly on compact subsets of $G$.
\end{Lem}
\begin{proof}
Take an approximate unit $(x_j)_{j\in J}$ of $A$.
Fix $f\in P_c(G)$
and set $y_j := \int_{G} f(g) \alpha_g(x_j) \,dm(g)$ for $j\in J$.
Then $(y_j)_{j\in J}$ is an approximate unit of $A$.
Moreover the net $(y_j)_{j\in J}$
is equi-$G$-continuous.
(That is, the maps $G \ni g \mapsto \alpha_g(y_j)$; $j\in J$, are equicontinuous in norm.)
Since $\alpha_g(y_i) - y_i \rightarrow 0$ weakly for all $g\in G$,
by the Hahn--Banach theorem (cf.~the construction of a quasi-central approximate unit),
one has an approximate unit $(e_i)_{i\in I}$ in the convex hull of $\{ y_j: j\in J\}$
satisfying $\|\alpha_g(e_i) - e_i\| \rightarrow 0$ for all $g\in G$.
As $(e_i)_{i\in I}$ is equi-$G$-continuous,
in fact the convergence is uniform on compact subsets of $G$.
\end{proof}
For two elements $x, y$ in a normed space $X$,
denote by $x\approx_\ve y$ when $\|x -y \| < \ve$.
\begin{Thm}\label{Thm:AP}
The QAP is equivalent to the Exel--Ng approximation property.
\end{Thm}
\begin{proof}
Clearly the QAP implies the Exel--Ng approximation property.

Assume that $(A, \alpha)$ is a $G$-\Cs-algebra with the Exel--Ng approximation property.
Choose a $G$-approximate unit $(e_i)_{i\in I}$ of $A$.
Take a bounded net $(\xi_j)_{j\in J}$ in $C_c(G, A) \subset L^2(G, A)$ that witnesses
the Exel--Ng approximation property of $\alpha$.
Put
\[c:= \sup_{j\in J} \|\xi_j\|^2.\]

Let a compact set $K \subset G$, a finite set $F\subset (A)_1$, and $\ve>0$ be given.
Take $i \in I$ satisfying
$e_i a \approx_\ve a \approx_{\ve} a e_i$, $\alpha_g(e_i) \approx_{\ve/c} e_i$ for $a\in F$, $g \in K$.
Choose $j\in J$ with
\[\langle \xi_j, x \tilde{\alpha}_g(\xi_j) \rangle \approx_\ve x\quad {\rm for}~ x\in \{e_i ^2, e_ia e_i, e_i a^\ast a e_i: a\in F\},~ g\in K \cup \{ e \}.\]
Put $\zeta:= e_i \xi_j$.
Observe that
\[\langle \zeta, \zeta \rangle = \langle \xi_j, e_i^2 \xi_j \rangle \approx_{\ve} e_i^2.\]
In particular $\|\zeta\|^2 < 1+ \ve$.
For any $g\in K$,
\[\langle \zeta, \tilde{\alpha}_g(\zeta) \rangle = \langle e_i \xi_j, \alpha_g(e_i)\tilde{\alpha}_g(\xi_j) \rangle
\approx_{\ve} \langle e_i \xi_j, e_i\tilde{\alpha}_g(\xi_j)\rangle
= \langle \xi_j, e_i ^2\tilde{\alpha}_g(\xi_j)\rangle
\approx_{\ve} e_i^2 \approx_{ \ve} \langle \zeta, \zeta \rangle.
\]
For any $a \in F$,
\[\langle \zeta a, a \zeta \rangle \approx_{\ve} a^\ast e_i a e_i \approx_{2\ve} a^\ast e_i^2a
\approx_{\ve} \langle \zeta a, \zeta a\rangle ,\]
\[\langle a \zeta, a \zeta \rangle =\langle \zeta, a^\ast a \zeta \rangle \approx_{\ve} e_i a^\ast a e_i \approx_{2\ve} e_i a^\ast e_i a
\approx_{\ve} \langle a\zeta, \zeta a\rangle.\]
These relations yield
\[\|[\zeta, a]\|^2 < 8 \ve.\]
This proves the QAP of $\alpha$.
\end{proof}
\begin{Rem}
The proof also shows that the third condition in the definition of the QAP can be replaced by the following (formally stronger) condition.
\[\ip{\xi_i, \xi_i} \rightarrow 1 {\rm~strictly~and~} \|\xi_i - \tilde{\alpha}_g(\xi_i)\| \rightarrow 0{\rm~uniformly~on~compact~subsets~of~}G.\]
\end{Rem}

\section{Characterizations of amenable \Cs-dynamics}
In this section, we show that
amenability is equivalent to the QAP for \emph{general} \Cs-dynamical systems.
This closes the circle of implications that were left open (cf.~ Section 8.1 of \cite{BEW2}).
We remark that the statement was recently shown
by Buss--Echterhoff--Willett \cite{BEW2} (see also \cite{ABF}) for \emph{unital discrete} \Cs-dynamics.
However we emphasize that, even if one is interested only in unital discrete \Cs-dynamics, the generalities here are crucial in Section \ref{Sec:Example}.

The following lemma replaces a uniform convergence condition by
a pointwise convergence condition (cf.\ Lemma~3.5 in \cite{BC}).
\begin{Lem}\label{Lem:conv}
Let $(A, \alpha)$ be a $G$-\Cs-algebra.
Let $(\xi_i)_{i \in I}$ be a bounded net in $L^2(G,A)$. 
Then $(\xi_i-g\ast \xi_i)a \to0$ for every $a\in A$ and $g\in G$ 
uniformly on compact subsets of $G$ if and only if 
$(\xi_i-f\ast \xi_i)a \to0$ for every $a\in A$ and $f\in P_c(G)$.
\end{Lem}
\begin{proof}
The `only if' part follows from direct calculations.
 We prove the converse. 
Let a compact subset $K\subset G$, a finite subset $F \subset (A)_1$, and $\ve>0$ be given. 
Take $f\in P_c(G)$.
Define \[U:=\{u\in G: \|u \ast f - f \|_1<\ve\}.\]
Choose a finite sequence $s_0:=e, s_1,\ldots,s_n$ in $G$ 
satisfying $K\subset \bigcup_{m=0}^n s_mU$. 
We put $f_m:=s_m \ast f \in C_c(G)$ for $m=0, 1, \ldots, n$. 
Assume that $\xi\in (L^2(G,A))_1$ satisfies
$\|(\xi-f_m \ast \xi)a\| <\ve$ for $m=0, 1, \ldots, n$, $a\in F$.
Then for $a\in F$ and $g\in K$, taking $u\in U$ and $m$ with $g= s_m u$,
one has
\[(g \ast \xi)a =[(s_mu) \ast \xi]a\approx_\ve [(s_m u)\ast f\ast \xi]a \approx_\ve (f_m\ast \xi)a \approx_\ve \xi a.
\]
This proves the claim.
\end{proof}

Now we are able to prove the promised theorem.
For a later application (Corollary \ref{Cor:closed}), we also
give a new characterization of amenable actions by the second dual.
\begin{Thm}\label{Thm:QAP}
Let $(A, \alpha)$ be a \Cs-dynamical system.
The following conditions are equivalent.
\begin{enumerate}[\upshape(1)]
\item The action $\alpha$ has the QAP.
\item The action $\alpha$ is amenable.
\item There is a $G$-conditional expectation
\[L^\infty(G) \btimes \cZ(A\sd )\rightarrow \cZ(A\sd).\]
Here we equip $L^\infty(G) \btimes \cZ(A\sd)$
with the diagonal $G$-action $\tilde{\alpha}$ of the left translation action and the $($possibly discontinuous$)$ action induced from $\alpha$.
\end{enumerate}
\end{Thm}
\begin{proof}
As the statement is trivial for compact groups,
we may assume that $G$ is non-compact.

(3) $\Rightarrow$ (2): As $\cZ(A''_\alpha)$ is a $G$-invariant (unital) direct summand of $\cZ(A\sd)$,
this is clear (by Theorem 1.1 of \cite{BC}).

(2) $\Rightarrow$ (1): We first prove that the QAP is equivalent to the following formally weaker condition. 
To ease the notation, write $\rQ(\xi):=\ip{\xi,\xi}$ for $\xi\in L^2(G,A)$ and 
define for $\vp \in \mathfrak{S}(A)$ a seminorm $\| \cdot \|_\varphi$ 
on $L^2(G, A)$ by $\|\xi\|_{\vp}:=\vp(\rQ(\xi))^{1/2}$.

\begin{claim}
A $G$-\Cs-algebra $(A, \alpha)$ has the QAP if $($and only if$)$ it satisfies
the following condition:

\begin{enumerate}
\item[$(\bigstar)$] For every $\vp \in \mathfrak{S}(A)$, 
every finite subsets $E\subset A$, $F\subset P_c(G)$, 
and every $\ve>0$, there is $\xi\in C_c(G,A)$ that satisfies 
\[
\|\xi\| \le1,\quad \|\xi\|_{\vp}>1-\ve,\quad \max_{(a, f)\in E\times F}\|(\xi-f\ast \xi)a\|_{\vp}<\ve,\quad 
\max_{a\in E}\|[\xi,a]\|_{\vp}<\ve.
\]
\end{enumerate}
\end{claim}
\begin{proof}[Proof of the Claim]
Assume that $(A,\alpha)$ satisfies condition ($\bigstar$) in the Claim. 
Then, by the Hahn--Banach theorem, for every finite subsets $E\subset A$ and $F\subset P_\mathrm{c}(G)$,
the norm closed convex hull of the set
\[
\left\{ \left(( a^\ast (1-\rQ(\xi))a )_{a\in E}, (\rQ((\xi-f\ast \xi)a))_{(a, f)\in E\times F}, (\rQ([\xi,a]))_{a\in E}\right): \xi\in (L^2(G,A))_1\right\}
\] 
in $\ell^\infty(E \sqcup (E \times F) \sqcup E, A)$
contains $0$.
Hence one has functions $\xi_1, \ldots, \xi_n \in C_c(G, A)$ and $\lambda_1, \ldots, \lambda_n \in [0, 1]$ with $\sum_{i=1}^n \lambda_i =1$
satisfying
\[\sum_{i=1}^n \lambda_i Q(\xi_i) \leq 1,\quad
\sum_{i=1}^n \lambda_i a^\ast(1- Q(\xi_i))a \leq \ve, \quad
\sum_{i=1}^n \lambda_i Q((\xi_i -f\ast \xi_i)a)\leq \ve, \quad
\sum_{i=1}^n \lambda_i Q([\xi_i, a]) \leq \ve\]
for all $a\in E$, $f\in F$.
Put $L:= \{e\} \cup [\bigcup_{f\in F} \supp(f)]$.
Take $t_1, \ldots, t_n \in G$ satisfying
\[[L \cdot \supp(\xi_i)\cdot t_i^{-1}] \cap [L \cdot \supp(\xi_j)\cdot t_j^{-1} ]= \emptyset\]
 for all $i \neq j$.
Define $\xi\in C_c(G, A)$ to be
\[\xi(g):= \sum_{i=1}^n \lambda_i^{1/2}\Delta(t_i)^{1/2} \xi_i(g t_i),~ g\in G.\]
Then, by routine calculations, for any $a\in E$ and $f\in F$, we have
\[Q(\xi) \leq 1,\quad a^\ast(1- Q(\xi))a \leq \ve,\quad
Q((\xi-f\ast \xi)a)\leq \ve,\quad Q([\xi, a]) \leq \ve.\]
Indeed one has
$Q(\xi)= \sum_{i=1}^n \lambda_i Q(\xi_i)$,
and similar for the other three.
By Lemma \ref{Lem:conv}, this shows the QAP of $\alpha$.
\end{proof}

Now assume $(A, \alpha)$ is amenable.
We verify that $(A, \alpha)$ satisfies condition ($\bigstar$) in the Claim.
Let $\varphi \in \mathfrak{S}(A)$, finite subsets $E\subset (A)_1$, 
$F\subset P_\mathrm{c}(G)$, and $\ve>0$ be given.
Take a compact neighborhood $V$ of $e\in G$ satisfying
\[\|a-\alpha_s(a)\|<\ve,\quad 
\|f-s^{-1} \ast f \ast s \|_1<\ve\]
for all $s\in V$, $a\in E$, $f\in F$.
Define $\vp_V \in \mathfrak{S}(A)$ to be
\[\vp_V(a):=\frac{1}{m(V)}\int_V\vp(\alpha_s (a))\,dm(s),\quad a\in A.\] 
Note that by Lemma \ref{Lem:contist}, $\vp_V$ defines a normal state on $A''_\alpha$.
Then, as $(A''_\alpha, \alpha)$ is amenable, by Theorem 1.1 of \cite{BC} and Kaplansky's density theorem (see Section \ref{Subsec:Wst}),
one can choose $\eta\in C_c(G, A)$ satisfying
\[\|\eta\|\le1,\]
\[1-\|\eta\|_{\vp_V}^2
 +\sum_{f\in F}\|\eta-f \ast \eta\|_{\vp_V}^2 
 +\sum_{a\in E}\|[\eta,a]\|_{\vp_V}^2
 < \ve^2.
\]
Hence one can find $s\in V$ satisfying the same inequality 
but $\vp_V$ being replaced by $\vp \circ \alpha_s$. 
Then, putting $\xi:=s\ast \eta \in C_c(G,A)$, one has 
\[\|\xi\|\le1 ,\quad \|\xi\|_{\vp} = \| \eta\|_{\vp \circ \alpha_s}> 1- \ve,\] 
\begin{align*}\|\xi-f \ast \xi\|_{\vp} &=\|\eta-s^{-1}\ast f \ast s \ast \eta\|_{\vp \circ \alpha_s} \\
&\leq \| (s^{-1}\ast f \ast s - f) \ast \eta\| + \|\eta-f \ast \eta\|_{\vp \circ \alpha_s}\\ &<2\ve \quad \quad {\rm for~}f\in F,
\end{align*}
\[\|[\xi,a]\|_{\vp} = \|[\eta, \alpha_s^{-1}(a)]\|_{\vp \circ \alpha_s} <3\ve \quad {\rm for~} a\in E.\]
This shows that $(A, \alpha)$ satisfies condition $(\bigstar)$ in the Claim. 

(1) $\Rightarrow$ (3): Assume that $\alpha$ has the QAP.
Take a net $(\xi_i)_{i\in I}$ in $C_c(G, A)$ as in the definition of the QAP.
For $\xi, \eta \in C_c(G, A)$,
we define a bounded linear map
\[\Theta_{\xi, \eta, \ast} \colon A^\ast \rightarrow L^1(G, \cZ(A\sd)_\ast)\]
to be 
\[[[\Theta_{\xi, \eta, \ast}(\varphi)](g)](x):= \varphi (\xi(g)^\ast x \cdot \eta(g))\quad{\rm~for~}\varphi \in A^\ast,~x\in \cZ(A\sd),~ g\in G.\]
Via the canonical identification
\[L^1(G, \cZ(A\sd)_\ast) = (L^\infty(G)\btimes \cZ(A\sd))_\ast\]
(\cite{Tak1}, Chapter 4, Theorem 7.17), the dual map defines
a normal bounded map
\[\Theta_{\xi, \eta} \colon L^\infty(G)\btimes \cZ(A\sd) \rightarrow A\sd.\]
For $F\in L^\infty(G) \odot \cZ(A\sd)$,
direct calculations show that
\[\Theta_{\xi, \eta}(F) =\int_G \xi(g)^\ast F(g) \eta(g)\,dm(g).\]
Then, by the Cauchy--Schwarz inequality,
we have
\[\|\Theta_{\xi, \eta}(F)\| \leq \|\xi\| \|\eta\| \|F\|.\]
Since $\Theta_{\xi, \eta}$ is normal, Kaplansky's density theorem yields
\[\|\Theta_{\xi, \eta}\| \leq \|\xi\| \|\eta\|.\]
A similar argument shows that, when $\eta= \xi$,
the map $\Theta_{\xi, \xi}$ is (completely) positive.

Now take a net $(\xi_i)_{i\in I}$ which witnesses the QAP of $\alpha$.
For each $i\in I$, put \[\Phi_i:= \Theta_{\xi_i, \xi_i} \colon L^\infty(G)\btimes \cZ(A\sd) \rightarrow A\sd.\]
Let $\Phi$ be a point-ultraweak cluster point of the net $(\Phi_i)_{i\in I}$.
We show that $\Phi$ is a $G$-conditional expectation
$L^\infty(G)\btimes \cZ(A\sd) \rightarrow \cZ(A\sd)$.
For any $a\in A$ and $F\in L^\infty(G) \odot \cZ(A\sd)$,
we have
\[[\Phi_i(F), a]= \Theta_{\xi_i, \xi_i a}(F) - \Theta_{ \xi_i a^\ast, \xi_i}(F).\]
By the normality of the appearing maps, the equation is still valid for all $F\in L^\infty(G) \btimes \cZ(A\sd)$.
As $\|[\xi_i, a]\|, \|[\xi_i, a^\ast]\| \rightarrow 0$, we have
\[\|\Theta_{\xi_i, a\xi_i} - \Theta_{\xi_i, \xi_i a}\| \rightarrow 0,\quad
\|\Theta_{ \xi_i a^\ast, \xi_i} - \Theta_{a^\ast \xi_i, \xi_i} \| \rightarrow 0.\]
Note also that $\Theta_{a^\ast \xi_i, \xi_i}=\Theta_{\xi_i, a \xi_i}$.
Thus $\|[\Phi_i(F), a]\| \rightarrow 0$ for $F\in L^\infty(G)\btimes \cZ(A\sd)$ and $a\in A$.
Hence $\Phi$ maps into $\cZ(A\sd)$.
For $z\in \cZ(A\sd)$, we have
$\Phi_i(z)= z \ip{\xi_i, \xi_i} \rightarrow z$.
Hence $\Phi \colon L^\infty(G)\btimes \cZ(A\sd) \rightarrow \cZ(A\sd)$ is a conditional expectation.

Finally, let $s\in G$ be given.
For $F\in L^\infty(G) \odot \cZ(A\sd)$, $a\in (A)_1$, and $i\in I$,
we have
\begin{align*}
a^\ast \alpha_s(\Phi_i(\tilde{\alpha}_s^{-1}(F)))a
&= \int_G a^\ast \alpha_s(\xi_i(g))^\ast \alpha_s(\alpha_s^{-1}F(sg)) \alpha_s(\xi_i(g))a\,dm(g) \\
&=\int_G a^\ast \alpha_s(\xi_i(s^{-1}g))^\ast F(g) \alpha_s(\xi_i(s^{-1}g))a\,dm(g)\\
&=\Theta_{\tilde{\alpha}_s(\xi_i)a, \tilde{\alpha}_s(\xi_i)a}(F),
\end{align*}
while
\[a^\ast \Phi_i(F) a = \Theta_{\xi_i a, \xi_i a}.\] 
Thus, by Kaplansly's density theorem,
\[\|a^\ast [\Phi_i(F) - (\alpha_s\circ \Phi_i \circ \tilde{\alpha}_s^{-1})(F)]a\|
\leq 2\|(\tilde{\alpha}_{s}(\xi_i) - \xi_i) a \| \| F\|\]
for all $F\in L^\infty(G)\btimes \cZ(A\sd)$ and $a\in (A)_1$.
This shows that $\Phi$ is $G$-equivariant.
\end{proof}
\begin{Rem}
The proof of Theorem \ref{Thm:QAP} shows that
the uniform convergence condition in the definition of the QAP
can be relaxed to the pointwise one.
\end{Rem}
The following is a consequence of Theorem \ref{Thm:QAP}.
For exact groups, the statement is shown in \cite{BEW2}.
\begin{Cor}\label{Cor:closed}
Let $H$ be a closed subgroup of $G$.
Let $\alpha \colon G \curvearrowright A$ be an amenable \Cs-dynamical system.
Then the restriction action $\alpha|_H \colon H \curvearrowright A$ is also amenable.
\end{Cor}
\begin{proof}
The statement is clear for \emph{open} subgroups.
Hence it suffices to show the statement for compactly generated $G$.
In this case, by the Kakutani--Kodaira theorem \cite{KK},
one can find a compact normal subgroup $K \lhd G$
whose quotient group $G/K$ is second countable.
Put
\[\bar{H}:= H/(K \cap H),\ \bar{G}:= G/K.\]
Then as explained in page 65 in \cite{Fol},
one has a measure-class preserving Borel isomorphism
$\bar{H} \times (\bar{H} \backslash \bar{G}) \rightarrow \bar{G}$ of the form
\[\bar{H} \times (\bar{H} \backslash \bar{G}) \ni (\bar{h}, x) \mapsto \bar{h} \cdot s(x) \in \bar{G}.\]
This gives rise to a normal unital $\bar{H}$-embedding
\[\Phi\colon L^\infty(\bar{H}) \rightarrow L^\infty(\bar{G}).\]
Let $\Psi \colon L^\infty(H) \rightarrow L^\infty(\bar{H})$ denote
the $H$-conditional expectation given by averaging by the right $(K\cap H)$-action with respect to the Haar measure on $K\cap H$.
Let $\iota \colon L^\infty(\bar{G}) \rightarrow L^\infty(G)$ be the obvious $G$-embedding.
By Theorem \ref{Thm:QAP},
one has a $G$-conditional expectation
\[\Theta \colon L^\infty(G) \btimes \cZ(A\sd )\rightarrow \cZ(A\sd).\]
The composite
\[\Theta \circ [(\iota \circ \Phi \circ \Psi) \btimes \id_{\cZ(A\sd)}] \colon L^\infty(H) \btimes \cZ(A\sd )\rightarrow \cZ(A\sd)\]
gives an $H$-conditional expectation.
\end{proof}

The following proposition gives a useful obstruction of amenability of \Cs-dynamical systems. Note that the statement is shown in \cite{BEW2} for discrete groups.
\begin{Prop}
Let $(A, \alpha)$ be a $G$-\Cs-algebra.
If $\alpha$ is amenable, then the associated action $G \curvearrowright \mathfrak{S}(A)$ is amenable.
\end{Prop}
\begin{proof}
Define a $G$-c.c.p.~ map $\Phi \colon A \rightarrow C_0(\mathfrak{S}(A))$ to be
 \[[\Phi(a)](\varphi):=\varphi(a) \quad {\rm for}~ a\in A, \varphi \in \mathfrak{S}(A).\]
Then by Dini's theorem, $\Phi$ is non-degenerate.
Take a net $(\xi_i)_{i \in I} \subset C_c(G, A)$
which witnesses the QAP of $\alpha$.
Since $\Phi$ is a non-degenerate $G$-c.c.p.~ map, the functions
\[k_i(g):= \Phi(\langle \xi_i, \tilde{\alpha}_g(\xi_i) \rangle);\quad g\in G,~ i\in I\]
in $C_c(G, C_0(\mathfrak{S}(A)))$ witness the amenability of $G \curvearrowright \mathfrak{S}(A)$ (\cite{AD02}, Proposition 2.5).
\end{proof}
We conclude the non-commutative analogue of Theorem 7.2 in \cite{AD02}.
\begin{Cor}
If $G$ admits an amenable action on a unital \Cs-algebra,
then $G$ is exact.
\end{Cor}

For later use, we record the following permanence properties of amenability of \Cs-dynamics
from \cite{BEW2}. 
They also follow rather immediately from the characterization (3) in 
Theorem~\ref{Thm:QAP} (togather with Remark \ref{Rem:tensor} below), except that the first one by the QAP. 
For the last statement, note that for any amenable normal closed subgroup $N \triangleleft G$, 
the natural embedding of $L^\infty(G/N)$ into $L^\infty(G)$ 
admits a $G$-equivariant conditional expectation. 
\begin{Prop}\label{Prop:perm}
Amenability of $G$-\Cs-algebras is inherited to
$G$-inductive limits, $G$-hereditary \Cs-subalgebras, $G$-extensions, $G$-quotients, 
and the range of $G$-conditional expectations. 

For a \Cs-dynamical system $\alpha\colon G\curvearrowright A$ 
and an amenable normal closed subgroup $N\triangleleft G$ with $N\subset\ker\alpha$, 
the induced action $G/N \curvearrowright A$ is amenable if and only if $\alpha$ is amenable. 
\end{Prop}
\begin{Rem}\label{Rem:tensor}
Here we recall the following general fact about c.c.p.~ maps
and von Neumann algebraic tensor product.
Let $\vp\colon M_1\to M_2$ be a
c.c.p.~ map
between von Neumann algebras and $N$ be another von Neumann algebra.
It is well-known that, when $\vp$ is normal, the (algebraic) tensor product map
$\vp\odot \id_N \colon M_1 \odot N \rightarrow M_2 \odot N$
uniquely extends to a normal c.c.p.~ map
\[\vp\btimes\id_N \colon M_1 \btimes N \rightarrow M_2 \btimes N.\]
Maybe less well-known is that, even when $\vp$ is not normal, there is a
c.c.p.~ map $\vp\btimes\id_N\colon M_1\btimes N\to
M_2\btimes N$ that
is uniquely determined by the condition
\[
(\id_{M_2}\btimes
\psi)(\vp\btimes\id_N)=(\vp\otimes\id_{\IC})(\id_{M_1}\btimes \psi)
\colon M_1\btimes N \to M_2 \otimes \IC\cong M_2
\]
for every $\psi\in N_\ast$.
This follows from the identification $M_i \btimes N=\mathrm{CB}(N_\ast, M_i)$ as an
operator space.
Alternatively, in the case where $N=\IB(\ell^2(I))$, this follows from
the identification
of $M\btimes\IB(\ell^2(I))$ and the space of $M$-valued $I\times I$
matrices whose
finite submatrices have uniformly bounded norms.
For a general von Neumann algebra $N\subset\IB(\ell^2(I))$, the map
$\vp\btimes\id_{\IB(\ell^2(I))}$
restricts to the desired map, thanks to Tomita's tensor commutant theorem.

Now let $M_3$, $N_1$, $N_2$ be other von Neumann algebras.
It is not difficult to see from the above condition that
\[
(\vp\btimes\id_{N_2})(\id_{M_1}\btimes \theta)=(\id_{M_2}\btimes
\theta)(\vp\btimes\id_{N_1})
\]
for any \emph{normal} c.c.p.~ map $\theta\colon N_1\to N_2$.
It is also clear that
\[(\vp_2\btimes\id_N)\circ (\vp_1\btimes\id_N)=(\vp_2\circ \vp_1)\btimes\id_N\]
for \emph{any} c.c.p.~ maps $\vp_i \colon M_i \rightarrow M_{i+1}$; $i=1, 2$.
Thus when $M_1, M_2, N$ are $G$-\Ws-algebras
and $\varphi \colon M_1 \rightarrow M_2$ is a $G$-c.c.p.~ map,
the map $\varphi \btimes \id_N$ is $G$-equivariant with respect to the diagonal actions.
\end{Rem}

\section{Characterizations through central sequence algebras}
In the context of classification/structure theory of operator algebras,
it is more natural and desirable to characterize a property
of non-commutative dynamical systems via the (relative) central sequence algebras.
This approach appears in Connes's successful classification of automorphisms on injective factors \cite{Con}.

For the QAP, such (formally stronger) reformulations are
obtained in \cite{Suz20c}, Theorem C, in some important cases.
These properties in fact play fundamental roles in the classification
of amenable \Cs-dynamical systems \cite{Suz20c}.
Here we establish such reformulations for general exact group \Cs-dynamics.

In the usage of central sequence algebras, it is natural to restrict
the attentions to second countable groups and separable \Cs-algebras.
Therefore, in this section, we only consider such cases.

\subsection*{Kirchberg's central sequence algebra}
Here we recall Kirchberg's (relative) central sequence algebras \cite{Kircent}.
Let $A$ be a \Cs-algebra.
Denote by $A^\infty$ the limit algebra of $A$ on $\mathbb{N}$:
\[A^\infty:= \ell^\infty(\IN, A) / c_0(\IN, A).\]
Set $A_\infty := A^\infty \cap A'$.
Here we regard $A$ as a \Cs-subalgebra of $A^\infty$
by the diagonal embedding.
Let $B \subset A^\infty$ be a \Cs-subalgebra.
Then define
\[\mathrm{Ann}(A^\infty, B):=\{ x \in A^\infty : x \cdot B = 0 = B \cdot x \}.\]
Observe that $\mathrm{Ann}(A^\infty, B)$ is an ideal of $A^\infty \cap B'$.
We set
\[F_{\infty}(B, A) := (A^\infty \cap B')/\mathrm{Ann}(A^\infty, B).\]
For a sequence $(x_n)_{n=1}^\infty$ in the preimage of $A^\infty \cap B'$,
we denote by $[x_n]_n^\infty$ its image in $F_\infty(B, A)$.
Observe that
\[ \| [x_n]_n^\infty\|_{F_{\infty}(B, A)} = \sup \{ \| [(x_n)_{n=1}^\infty+ c_0(\IN, A)] \cdot b \|_{A^\infty}: b\in (B)_1\}.\]
When $B$ is $\sigma$-unital, a standard reindexation argument shows that
$F_{\infty}(B, A)$ is unital \cite{Kircent}.
If $(A, \alpha)$ is a $G$-\Cs-algebra and $B \subset A^\infty$ is invariant under
the diagonal $G$-action, then we denote by
$F_{\infty, \alpha}(B, A)$ the set of
all $G$-continuous elements in $F_{\infty}(B, A)$ \cite{SzI}.
Note that $F_{\infty, \alpha}(B, A)$, equipped with the restriction action, is a $G$-\Cs-algebra.
Put $F_{\infty}(A) := F_{\infty}(A, A)$, $F_{\infty, \alpha}(A):= F_{\infty, \alpha}(A, A)$ for short.
For notational simplicity,
we reuse the same symbol $\alpha$ for the induced $G$-actions on $A^\infty$, $F_{\infty, \alpha}(B, A)$.
\begin{Prop}\label{Prop:Gucp}
Let $G$ be second countable.
Let $\alpha \colon G \curvearrowright A$ be an amenable action on a separable \Cs-algebra.
Then for any unital $G$-\Cs-algebra $(C, \gamma)$,
there is a $G$-u.c.p.~ map $C \rightarrow F_{\infty, \alpha}(A)$.
\end{Prop}
\begin{proof}
Take a sequence $(\xi_n)_{n=1}^\infty \in C_c(G, A) \subset L^2(G, A)$
as in the definition of the QAP.
Fix $\varphi\in \mathfrak{S}(C)$.
For each $n\in \mathbb{N}$,
define $\Phi_n \colon C \rightarrow A$ to be
\[\Phi_n(c):= \int_{G} \varphi(\gamma_g^{-1}(c)) \xi_{n}(g)^\ast \xi_{n}(g)\,dm(g),\quad c\in C.\]
We show that the sequence $(\Phi_n)_{n=1}^\infty$ defines a $G$-u.c.p.~ map
$\Phi_\infty \colon C \rightarrow F_{\infty, \alpha}(A)$.
Observe that for $a\in A$, $c\in C$, and $n\in \mathbb{N}$, the Cauchy--Schwarz inequality implies
\[\|[a, \Phi_n(c)] \| \leq 2\|c\| \|[a, \xi_n]\| .\]
Thus the sequence $(\Phi_n)_{n=1}^\infty$
defines a c.c.p.~ map $\Phi_\infty \colon C \rightarrow F_\infty(A)$
by
\[\Phi_\infty(c):= [\Phi_n(c)]_n^\infty,\quad c\in C.\]
Since
$\Phi_n(1)= \langle \xi_n, \xi_n \rangle$,
the map
$\Phi_\infty$ is unital.
For $c\in C$ and $s\in G$,
\begin{align*}
\alpha_s(\Phi_n(\gamma_s^{-1}(c))) &= \int_{G} \varphi(\gamma_{sg}^{-1}(c))\alpha_s(\xi_{n}(g))^\ast \alpha_s(\xi_{n}(g)) \,dm(g)\\
&=\int_G \varphi(\gamma_g^{-1}(c))[\tilde{\alpha}_s(\xi_{n})(g)]^\ast\tilde{\alpha}_s(\xi_{n})(g)\,dm(g).
\end{align*}
Hence
\[\| a^\ast( \alpha_s (\Phi_n(c)) - \Phi_n(\gamma_s(c)))a\|
\leq 2\|c \| \| (\tilde{\alpha}_s(\xi_n) - \xi_n )a\|\]
for all $a\in (A)_1$ and $c\in C$.
This shows that $\Phi_\infty$ is $G$-equivariant.
As $\Phi_\infty$ is continuous and $C$ is a $G$-\Cs-algebra,
we obtain
$\Phi_\infty(C) \subset F_{\infty, \alpha}(A)$.
\end{proof}

\begin{Def}
Let $(A, \alpha)$ be a $G$-\Cs-algebra.
We say that an element $(x_n)_{n=1}^\infty \in \ell^\infty(\IN, A)$ is
\emph{equi-$G$-continuous in the strict topology},
if for any $a\in A$,
the maps
\[G \ni g \mapsto \alpha_g(x_n) a;~ n\in \IN\]
are equicontinuous in norm.
Similarly, for a bounded sequence $(\xi_n)_{n=1}^\infty \subset L^2(G, A)$,
we say the sequence is \emph{equi-$G$-continuous in the strict topology}
if for any $a\in A$,
the maps \[G \ni g \mapsto \tilde{\alpha}_g(\xi_n) a;~ n\in \IN\]
are equicontinuous in norm.
\end{Def}

\begin{Lem}\label{Lem:lift}
Let $G$ be second countable.
Let $(A, \alpha)$ be a separable $G$-\Cs-algebra.
\begin{enumerate}[\upshape(1)]
\item
Let $a \in F_{\infty, \alpha}(A)$.
Then any its representing sequence $(a_n)_{n=1}^\infty \in \ell^\infty(\mathbb{N}, A)$ 
is equi-$G$-continuous in the strict topology.
\item
Let $\xi \in L^2(G, F_{\infty, \alpha}(A))$.
Then for any $\ve>0$,
there is a bounded sequence $(\eta_n)_{n=1}^\infty \in C_c(G, A)$
as follows.
\begin{itemize}
\item $(\eta_n)_{n=1}^\infty$ is equicontinuous in norm, and equi-$G$-continuous in the strict topology,
\item $\eta_n$'s are supported in a same compact set,
\item $\sup\{ \|\eta_n(g)\|: n \in \IN,~g\in G\} < \infty$,
\item
the map $\eta_\infty(g):= [\eta_n(g)]_n^\infty$; $g\in G$,
sits in $C_c(G, F_{\infty, \alpha}(A))$ and satisfies
$\eta_\infty \approx_{\ve} \xi$.
\end{itemize}
\end{enumerate}
\end{Lem}

\begin{proof}
(1): This is Lemma 2.2 in \cite{SzI}.

(2): We may assume that $\xi= \sum_{i=1}^N \rho_i \otimes x_i$
for some $\rho_i \in C_c(G)$ and $x_i \in F_{\infty, \alpha}(A)$.
For each $x_i$, take a representing sequence $(x_{i, n})_{n=1}^\infty\in \ell^\infty(\mathbb{N}, A)$.
Set \[\eta_n:=\sum_{i=1}^N \rho_i \otimes x_{i, n} \in C_c(G, A).\]
Then, by (1), this gives the desired sequence.
\end{proof}

\begin{Thm}\label{Thm:central}
Let $G$ be a second countable exact group.
Let $\alpha \colon G \curvearrowright A$ be a \Cs-dynamical system
on a separable \Cs-algebra.
Then the following conditions are equivalent.
\begin{enumerate}[\upshape(1)]
\item The action $\alpha$ is amenable.
\item For any unital $G$-\Cs-algebra $C$,
there is a $G$-u.c.p.~ map $C \rightarrow F_{\infty, \alpha}(A)$.
\item There is a sequence $(\xi_n)_{n=1}^\infty$ in $L^2(G, F_{\infty, \alpha}(A))$
satisfying
$\langle \xi_n, \tilde{\alpha}_g(\xi_n)\rangle \rightarrow 1$ in norm,
uniformly on compact subsets of $G$.
\item There is a sequence $(k_n)_{n=1}^\infty \subset C_c(G, F_{\infty, \alpha}(A))$
of positive type functions
converging to $1$ in norm uniformly on compact subsets of $G$.
\item The action $\alpha \colon G \curvearrowright F_{\infty, \alpha}(A)$
 is amenable.
\item For any $G$-invariant separable \Cs-subalgebra $B \subset A^\infty$,
the action $\alpha \colon G \curvearrowright F_{\infty, \alpha}(B, A)$
is amenable.
\end{enumerate}
\end{Thm}
\begin{Rem}
In a private communication, the authors are informed from Siegfried Echterhoff
that he, together with Alcides Buss, Rufus Willett, has independently obtained the same statement for unital discrete \Cs-dynamics.
\end{Rem}
\begin{proof}[Proof of Theorem \ref{Thm:central}]
(1) $\Rightarrow$ (2): 
This is Proposition \ref{Prop:Gucp}.

(2) $\Rightarrow$ (4): Let $G \curvearrowright X$ be an amenable action
on a compact space.
Choose a sequence $(k_n)_{n=1}^\infty \subset C_c(G, C(X))$
of positive type functions
which converges to $1$ uniformly on compact subsets of $G$.
Take a $G$-u.c.p.~ map $\Phi \colon C(X) \rightarrow F_{\infty, \alpha}(A)$.
Then the sequence $(\Phi \circ k_n)_{n=1}^\infty$ confirms condition (4).

(4) $\Rightarrow$ (3): This follows from Proposition 2.5 of \cite{AD}.

(3) $\Rightarrow$ (6):
Let a compact set $K \subset G$ and $\ve>0$ be given.
Choose $\xi \in C_c(G, F_{\infty, \alpha}(A))$
with $\|\xi \|=1$ and $\langle \xi, \tilde{\alpha_s}(\xi) \rangle \approx_\ve 1$ for $s\in K$.

Choose a bounded sequence $(\xi_n)_{n=1}^\infty \in C_c(G, A)$
as in Lemma \ref{Lem:lift} (2) such that
$\xi_\infty(g):= [\xi_n(g)]_n^\infty$ satisfies
$\xi_\infty \approx_\ve \xi$.
We may assume that $\|\xi_{\infty}\| \leq 1$.
Observe that for any $a\in A$, the functions
\[s\in K \mapsto \|a^\ast (\langle \xi_n, \tilde{\alpha}_s(\xi_n)\rangle -1)a\|;~n\in \IN\]
are equicontinuous.
Hence for any $a\in (A)_1$, one has
\[\limsup _{n \rightarrow \infty} \max_{s\in K} \|a^\ast (\langle \xi_n, \tilde{\alpha}_s(\xi_n)\rangle -1 )a\|< 3\ve.\]

Now let $B \subset A ^\infty$ be a separable $G$-invariant \Cs-subalgebra. 
Let $F =\{y_1, \ldots, y_m\} \subset (F_{\infty, \alpha}(B, A))_1$ be a finite subset.
Take a dense sequence $(x_i)_{i=1}^\infty$ in $(B)_1$.
For each $x_i$, $y_j$; $i \in \mathbb{N}$, $j=1, 2, \ldots, m$,
fix their representing sequences $(x_{i, n})_{n=1}^\infty$, $(y_{j, n})_n^\infty \in (\ell^\infty(\IN, A))_1$.
By the choice of $\xi_n$'s and the observation in the previous paragraph,
one can find a subsequence $(\xi_{l(n)})_{n=1}^\infty$
satisfying
\begin{align*}
\|[x_{i, n}, \xi_{l(n)}(g)]\| &< 1/n \quad{\rm~ for~all~} i \leq n,~ g\in G,~n\in \IN, \\
\|[y_{i, n}, \xi_{l(n)}(g)]\| &< 1/n \quad{\rm~for~ all~} i=1, \ldots, m,~g\in G,~n\in \IN, \end{align*}
\[\max_{s\in K} \| x_{i, n}^\ast (\langle\xi_{l(n)}, \tilde{\alpha}_s(\xi_{l(n)})\rangle -1 )x_{i, n}\| < 3\ve \quad{\rm~for~all~}i\leq n, ~n\in \IN.\]
Define the map $\tilde{\zeta} \colon G \rightarrow \ell^\infty(\IN, A)$ to be
\[\tilde{\zeta}(g):= (\xi_{l(n)}(g))_{n=1}^\infty,~ g\in G.\]
By the above three inequalities,
by passing to the quotient, $\tilde{\zeta}$ defines a function
$\zeta \in C_c(G, F_{\infty, \alpha}(B, A))$ satisfying
$[\zeta, y] =0$ for $y\in F$ and
$\langle \zeta, \tilde{\alpha}_s(\zeta) \rangle \approx_{3\ve} 1$ for $s\in K$.
This confirms condition (6).

(6) $\Rightarrow$ (5) $\Rightarrow$ (3): Trivial.

(3) $ \Rightarrow$ (1): Let $F \subset A$ be a finite set.
Let $e\in K \subset G$ be a compact subset.
Let $0< \ve< 1$.
Take $a \in (A)_{1}$ satisfying $a x \approx_\ve x \approx_{\ve} xa$ for all $x\in F$
and that $\alpha_s(a) \approx_\ve a$ for all $s\in K$ (Lemma \ref{Lem:qa}).

Choose $\xi \in C_c(G, F_{\infty, \alpha}(A))$ with
$\langle \xi, \tilde{\alpha}_s(\xi) \rangle \approx_\ve 1$ for $s\in K$,
$\|\xi\| \leq 1$.
Applying Lemma \ref{Lem:lift} (2) to $\xi$ and $\ve>0$,
one has a sequence $(\xi_n)_{n=1}^\infty \subset C_c(G, A)$ as in the statement.
By the choice of $(\xi_n)_{n=1}^\infty$, one has $m\in \mathbb{N}$ satisfying
$a^\ast \langle \xi_m, \tilde{\alpha}_s(\xi_m) \rangle \alpha_s(a) \approx _{3\ve} a^\ast a$ for all $s\in K$
and $\|[\xi_m, y]a\|< \ve$ for all $y \in aF$.
Set $\zeta:= \xi_m a$.
Then $\langle \zeta, \tilde{\alpha}_s(\zeta) \rangle \approx_{3\ve} a^\ast a$
for all $s\in K$. 
In particular $\|\zeta\| < \sqrt{1+3\ve} < 2$.
For any $x\in F$,
\[x\zeta \approx_{2\ve} ax \zeta = a x\xi_m a \approx_\ve \xi_m a xa = \zeta xa \approx_{2\ve} \zeta x.\]
This shows the QAP of $\alpha$.
\end{proof}
\begin{Rem}
When $G$ is discrete, one can replace `$\infty$' by any free ultrafilter $\omega$ on $\IN$.
In particular, the theorem removes the assumptions of Theorem C in \cite{Suz20c}
\end{Rem}
\begin{Rem}
Theorem \ref{Thm:central} does not hold true for non-exact groups.
To see this, consider a second countable non-exact group $G$ (\cite{Gro}, \cite{Os}).
Then the left translation action $G \curvearrowright C_0(G)$ is amenable,
but it fails conditions (3) to (6).
\end{Rem}
\section{On compact group fixed point algebras}\label{Sec:fixed point}
When one wants to show the nuclearity of a given \Cs-algebra,
the following theorem is quite powerful:
A \Cs-algebra is nuclear if it admits a compact group action 
whose fixed point algebra is nuclear.
The theorem is successfully applied to celebrated \Cs-algebras, including
the Cuntz algebras, irrational rotation algebras, graph algebras, Pimsner algebras. For details, we refer the reader to Section 4.5 of \cite{BO}.

Here we establish its \Cs-dynamics analogue.
This result is important in the next section, where we give a powerful framework to produce amenable actions on simple \Cs-algebras.
\begin{Thm}\label{Thm:fixed point}
Let $G$ be a locally compact group.
Let $K$ be a compact group. 
Let $(A, \alpha \times \beta)$ be a $(G \times K)$-\Cs-algebra.
Then the following conditions are equivalent.
\begin{enumerate}[\upshape (1)]
\item The action $G\times K\acts A$ is amenable.
\item The action $G\acts A$ is amenable.
\item The induced action $G\acts A\rtimes K$ is amenable.
\item The restriction action $G\acts A^K$ is amenable.
\end{enumerate}
\end{Thm}
The implication $(2)\Rightarrow(4)$ follows from Proposition \ref{Prop:perm}.
The implication $(1)\Rightarrow(2)$ follows from Corollary \ref{Cor:closed}.
We show the implications $(4)\Rightarrow(3) \Rightarrow(1)$,
which complete the proof. 
The proof of $(4)\Rightarrow(3)$ is by studying the Wassermann type inclusion \cite{Was}.

For this, we first recall some basic facts on compact group actions. 
Let $K$ be a compact group.
Let $\rho \colon K \curvearrowright L^2(K)$ denote the right regular representation.
Let $(A, \beta)$ be a $K$-\Cs-algebra. Consider 
the $K$-\Cs-algebra $(A\otimes\IK(L^2(K)), \beta\otimes\Ad_{\rho})$.
Then there is a natural identification $(A\otimes\IK(L^2(K)))^K= A\rtimes K$.
Since we need an explicit isomorphism, let us recall the proof.
\begin{proof}
We first recall that, by the Peter--Weyl theorem, ${\mathrm C}^\ast_\lambda(K)= \IK(L^2(K))^{\Ad_\rho}$. 
Define a $\ast$-homomorphism $\pi \colon A \rightarrow A\otimes C(K)$
to be
\[\pi(a)(k):=\beta_k^{-1}(a) \quad {\rm~for~}a\in A,~k\in K.\]
We regard $A\otimes C(K)$ as a \Cs-subalgebra of $\cM(A\otimes\IK(L^2(K)))$
in the obvious way. 
Then, one has $\pi(A)\subset \cM(A\otimes\IK(L^2(K)))^K$ and 
\[
A\otimes\IK(L^2(K)) = \spa [\pi(A)(1\otimes\IK(L^2(K)))].
\]
By applying the conditional expectation
$\int_K \beta_k \otimes \Ad_{\rho_k} dm(k)$ to both sides, we obtain
\[
(A\otimes\IK(L^2(K)))^K = \spa [\pi(A)(1\otimes {\mathrm C}^\ast_\lambda(K))] = A\rtimes K.\]
\end{proof}
\noindent
Proof of (4) $\Rightarrow$ (3).
Take an approximate unit $(p_i)_{i\in I}$ of projections in 
${\mathrm C}^\ast_\lambda(K) \subset \IK(L^2(K))$. 
Note that
\[(1 \otimes p_i)(A\otimes\IK(L^2(K)))^K (1 \otimes p_i)= (A\otimes\IK(p_iL^2(K)))^K\]
for $i \in I$.
Thus $A \rtimes K$ is (the closure of) the directed union of
\Cs-subalgebras $(A\otimes\IK(p_iL^2(K)))^K$, $i \in I$.
Notice that for each $i \in I$, the conditional expectation
\[\id_A\otimes\mathrm{tr} \colon (A\otimes\IK(p_iL^2(K)))^K \rightarrow A^K\otimes\IC p_i
\] 
has finite index.
Recall that a conditional expectation $E$ on a \Cs-algebra $B$
is said to have finite (probabilistic) index \cite{PP} if it satisfies $E\geq \lambda \cdot \id_B$ 
for some positive number $\lambda>0$.

Now let $(A, \alpha \times \beta)$ be a $(G\times K)$-\Cs-algebra.
Then all constructions above 
respect the $G$-action $\alpha$. 
Hence the implication $(4)\Rightarrow(3)$ follows from the next lemma.
\begin{Lem}
Let $(B, \beta)\subset (A, \alpha)$ be an inclusion of $G$-\Cs-algebras with 
a $G$-conditional expectation $E$ of finite index. 
If $\beta$ is amenable, then so is $\alpha$.
\end{Lem}

\begin{proof}
Denote by $\iota \colon B \rightarrow A$ the inclusion map.
Consider the normal extensions
$\iota'' \colon B''_{\beta} \rightarrow A''_\alpha$,
$E'' \colon A''_{\alpha} \rightarrow B''_{\beta}$ (see Lemma \ref{Lem:extension}).
Observe that $E'' \circ \iota''= \id_{B''_{\beta}}$.
We regard $B''_\beta$ as a $G$-\Ws-subalgebra of $A''_\alpha$ by $\iota''$.
Then $E''\colon A''_{\alpha} \rightarrow B''_{\beta}$ is a normal $G$-conditional expectation.
Since $E$ has finite index, so does $E''$.
Thus the statement follows from the next lemma.
\end{proof}

\begin{Lem}
Let $(N, \beta) \subset (M, \alpha)$ be an inclusion of $G$-\Ws-algebras 
with a normal $G$-conditional expectation $E$ of finite index. 
If $(N, \beta)$ is amenable, then so is $(M, \alpha)$.
\end{Lem}
\begin{proof}
For the following, consult 1.1.3 in \cite{Pop} (statements 
$6^\circ$ to $9^\circ$ hold true in general, by replacing a normal faithful state 
$\vp$ there with a normal faithful semifinite weight, which always exists \cite{Tak72}). 
Consider the basic construction of $E$:
\[\ip{M,e_N}=\IB(L^2(M))\cap (JNJ)'.\]
As $E$ is $G$-equivariant, the $G$-action on $M$ extends to a (continuous) $G$-action on $\ip{M, e_N}$
which fixes the Jones projection $e_N$.
Since $E$ is of finite index, 
there is a normal completely positive map
\[\widehat{E} \colon \ip{M,e_N} \rightarrow M\]
satisfying
$\widehat{E}(x e_N y)= xy$ for $x, y\in M$.
The map
\[E_1:= \widehat{E}(1_{\ip{M,e_N}})^{-1} \cdot \widehat{E} \colon \ip{M,e_N} \rightarrow M\]
defines a conditional expectation (called the dual conditional expectation; see Section 1.2 in \cite{Pop} for details).
Since $\cZ(N)\cong\cZ(\ip{M,e_N})$ via the $G$-$\ast$-isomorphism $z\mapsto Jz^\ast J$, the $G$-action on $\ip{M,e_N}$ is amenable.
Since $E_1$ is $G$-equivariant, this shows the amenability of $(M, \alpha)$.
\end{proof}

\begin{proof}[Proof of $(3)\Rightarrow(1)$]
We will give a sequence of $(G \times K)$-u.c.p.~ 
maps 
\[L^\infty(G \times K) \btimes A''_{\alpha\times\beta} \rightarrow L^\infty(G)\btimes A''_{\alpha\times\beta}
\rightarrow L^\infty(G)\btimes (A''_{\alpha\times\beta}\bar\rtimes K)
\to A''_{\alpha\times\beta}\bar\rtimes K
\to A''_{\alpha\times\beta}
\]
whose composite is a conditional expectation onto $A''_{\alpha\times\beta}$. 
Here we equip the \Ws-crossed product $A''_{\alpha\times\beta}\bar\rtimes K$ with the conjugate $K$-action.

The first map is given by
$\id_{L^\infty(G)} \btimes m_K \btimes \id_{A''_{\alpha \times \beta}}$,
where $m_K$ denotes the Haar (probability) measure on $K$.
The second map is the canonical inclusion.

To give the third map, we first take a $G$-conditional expectation 
\[\Psi \colon L^\infty(G) \btimes (A\rtimes K)''_\alpha \rightarrow (A\rtimes K)''_\alpha,\]
which exists by assumption.
As the $K$-action on $(A\rtimes K)''_\alpha$ is inner,
the map $\Psi$ is in fact $(G \times K)$-equivariant.
By the universal property of $(A\rtimes K)''_\alpha$, we have a $G$-invariant central projection $z\in (A\rtimes K)''_\alpha$ with
$A''_{\alpha\times\beta}\bar\rtimes K=z(A\rtimes K)''_\alpha$.
Therefore $\Psi$ restricts to a $(G \times K)$-conditional expectation
\[L^\infty(G) \btimes (A''_{\alpha\times\beta}\bar\rtimes K) \rightarrow A''_{\alpha\times\beta}\bar\rtimes K.\]

To construct the fourth map, choose a local basis $\mathcal{U}$ of $e\in K$
consisting of $K$-conjugation invariant sets. (Such a $\mathcal{U}$ exists by the Peter--Weyl theorem.)
Define $\zeta_U:= m(U)^{-1/2} \chi_{U} \in L^2(K)$ for $U\in \mathcal{U}$.
Take
any point-ultraweak cluster point of the $(G\times K)$-u.c.p.~ maps
$(\id_{A''_{\alpha\times\beta}}\otimes\omega_{\zeta_U})_{U\in \mathcal{U}}$.
Here we identify $A''_{\alpha\times\beta}\bar\rtimes K$
with the \Ws-subalgebra of $A''_{\alpha\times\beta}\btimes \IB(L^2(K))$
in the same way as noted after Theorem \ref{Thm:fixed point}.

Now it is clear from the definitions that the resulting composite 
 is indeed a $(G \times K)$-conditional expectation.
\end{proof}

\section{New constructions of non-commutative amenable actions}\label{Sec:Example}
\subsection{New constructions}
All the previously known constructions \cite{Suzeq} of (really) non-commutative amenable actions
depend on the reduced crossed product construction.
In this section we give a new construction based on the Pimsner construction \cite{Pim97} (see also \cite{Mey}).
Notable advantages of the present construction
are
\begin{itemize}
\item it works for arbitrary locally compact groups,
\item it is easy to control equivariant Kasparov theoretic data \cite{Kas},
\item it does not involve the infinite tensor product construction.
(In particular it provides simple amenable \Cs-dynamics even from \emph{proper} actions.)
\end{itemize}
We say an inclusion $A \subset B$ of \Cs-algebras
is \emph{non-degenerate} if $AB$ is dense in $B$.
\begin{Thm}\label{Thm:Toeplitz}
Let $(A, \alpha)$ be an amenable \Cs-dynamical system.
Then there is a non-degenerate ambient $G$-\Cs-algebra $(B, \beta)$ of $(A, \alpha)$ with the following properties.
\begin{itemize}
\item The action $\beta$ is amenable and outer.
\item The \Cs-algebra $B$ is purely infinite simple.
\item When $A$ is nuclear, so is $B$.
\end{itemize}
When $A$ is separable and $G$ is second countable,
one can further arrange $B$ to be separable, and the inclusion $A \subset B$
to be {\rm KK}$^G$-equivalent.
\end{Thm}
\begin{proof}
We show that Meyer's construction \cite{Mey} gives the desired $G$-\Cs-algebra $(B, \beta)$.
The statements except the amenability of $(B, \beta)$ are shown in the proof of Theorem 2.1 in \cite{Mey} (cf.~\cite{Pim97}, \cite{Kum}).
To show the amenability of the resulting $G$-\Cs-algebra, we first recall the construction of $(B, \beta)$.

Take a non-degenerate faithful $\ast$-representation $\pi \colon A \rightarrow \mathbb{B}(\cH)$ on a Hilbert space $\cH$ with $\pi(A) \cap \mathbb{K}(\cH)=0$.
When $A$ is separable, we may choose a separable $\cH$.
Define $\tilde{\pi} \colon A \rightarrow \mathbb{B}(L^2(G, \cH))$
to be
\[(\tilde{\pi}(a)\xi)(g):=\pi(\alpha_{g}^{-1}(a))(\xi(g)) \quad {\rm~for~}\xi \in C_c(G, \cH),~ a\in A,~ g\in G.\]
Set $\cE:= L^2(G, \cH) \otimes A$,
a right Hilbert $A$-module.
On $\cE$, we equip the left $A$-action $\varphi \colon A \rightarrow \IB(\cE)$ and the isometric $G$-action $\gamma$ as follows.
\[\varphi:= \tilde{\pi} \otimes 1,\quad \gamma:= (\lambda \otimes \id_{\cH}) \otimes \alpha.\]
Then $\varphi$ and $\gamma$
make $\cE$ a $G$-\Cs-correspondence over $A$.
By the choice of $\pi$,
\[I_{\cE}=\varphi(A) \cap \mathbb{K}(\cE) = 0.\]

From now on we freely use the facts and notations on the Toeplitz--Pimsner algebras
recorded in Section 4.6 of \cite{BO}.
Let $B:= \mathcal{T}_{\cE}$ be the Toeplitz--Pimsner algebra of $\cE$ \cite{Pim97}.
We equip $B$ with the $G$-action $\beta$ induced from $\gamma$;
 $\beta_g(T_\xi):= T_{\gamma_g(\xi)}$ for $g\in G$, $\xi \in \cE$.
The canonical embedding $A \rightarrow B$ is non-degenerate and $G$-equivariant.
We will show the amenability of $(B, \beta)$.

Let $\sigma \colon \IT \curvearrowright B$ be the gauge action.
Clearly $\sigma$ commutes with $\beta$.
By Theorem \ref{Thm:fixed point}, it suffices to show the amenability of
$(B^\sigma, \beta)$.

Recall that
\[B^\sigma = \overline{\rm span}\{ T_\xi T_\eta^\ast: \xi, \eta \in \cE^{\otimes n}, n\in\IZ_{\geq 0}\}.\]
For each $n\in\IZ_{\geq 0}$,
set \[B_{\leq n}:=\overline{\rm span} \{ T_\xi T_\eta^\ast: \xi, \eta \in \cE^{\otimes k}, k\leq n\},\quad B_{n}:=\overline{\rm span} \{ T_\xi T_\eta^\ast: \xi, \eta \in \cE^{\otimes n}\}.\]
Observe that both $B_{\leq n}$ and $B_{n}$ are $G$-\Cs-subalgebras of $B$.
Moreover $B^\sigma$ is (the closure of) the directed union of $B_{\leq n}$'s.
Therefore, by Proposition \ref{Prop:perm}, it suffices to show
the amenability of $(B_{\leq n}, \beta)$'s.

As $B_{\leq 0}= A$, the case $n=0$ is the assumption.
For a given $n\in \mathbb{N}$, assume that $B_{\leq n-1}$ is amenable.
Recall that we have a short exact sequence
\[0 \rightarrow B_{n} \rightarrow B_{\leq n} \rightarrow B_{\leq n-1} \rightarrow 0\]
which is obviously $G$-equivariant.
Observe that
$B_{n} \cong \mathbb{K}(\cE^{\otimes n})$ as $G$-\Cs-algebras.
Note that $\cE^{\otimes n} \cong L^2(G, \cH) ^{\otimes n} \otimes A$ as $G$-\Cs-correspondences over $A$.
(Indeed, when $n=2$, the map
\[[\xi \otimes a] \otimes _A [\eta \otimes b] \mapsto [\xi \otimes \tilde{\pi}(a)\eta]\otimes b;\quad \xi, \eta \in L^2(G, \cH), a, b \in A\]
extends to the desired isomorphism. The statement for general $n \in \IN$ is shown in a similar way.)
Therefore
\[(B_n, \beta) \cong (\mathbb{K}(L^2(G, \cH)^{\otimes n})\otimes A, \Ad_{(\lambda \otimes \mathrm{id}_{\cH})^{\otimes n}} \otimes \alpha).\]
This isomorphism shows the amenability of $(B_{n}, \beta)$.
\end{proof}
\subsection{Consequences}
As Theorem \ref{Thm:Toeplitz} is applied to many \Cs-dynamics,
we naturally obtain various new examples of \emph{simple} amenable \Cs-dynamics.
Here we list significant examples among them.

The next result improves the main results of \cite{Suzeq}.
This gives a more complete answer to questions of Anantharaman-Delaroche (\cite{AD02}, Question 9.2 (c), (d)).
\begin{Cor}
The following statements hold true.
\begin{enumerate}[\upshape(1)]
\item
For every non-amenable exact group $G$, there is a unital inclusion $(A, \alpha) \subset (B, \beta)$ of unital $G$-\Cs-algebras with the following properties:
\begin{itemize}
\item Both $A$ and $B$ are purely infinite simple, nuclear.
\item The action $\alpha$ is amenable.
\item $B \rtimes G \neq B \rc G$.
\end{itemize}
\item
For every non-amenable group $G$, there is a non-degenerate inclusion $(A, \alpha) \subset (B, \beta)$ of $G$-\Cs-algebras with the following properties:
\begin{itemize}
\item Both $A$ and $B$ are purely infinite simple, nuclear.
\item The action $\alpha$ is amenable.
\item $B \rtimes G \neq B \rc G$.
\end{itemize}
\end{enumerate}
\end{Cor}
\begin{proof}
We only show statement (1). Statement (2) is shown in 
the same way after using the proper action $G \curvearrowright G$ instead of $\alpha_0$ below.

Since $G$ is exact, by Proposition \ref{Prop:exact}, one can choose an amenable action $\alpha_0 \colon G \curvearrowright A_0$ on a unital nuclear \Cs-algebra.
Choose a unital faithful covariant representation $(\pi, u)$ 
of $A_0$ on a Hilbert space $\cH$.
Set $B_0:= \pi(A_0) + \mathbb{K}(\cH)$.
We equip $B_0$ with the $G$-action $\beta_0:= \Ad_u$.
Then as the restriction action $G \curvearrowright \mathbb{K}(\cH)$ is inner,
$B_0 \rtimes G \neq B_0 \rc G$ (see e.g., Lemma 3.1 of \cite{Suzeq}).
Applying the construction in Theorem \ref{Thm:Toeplitz}
to the $\ast$-representations $\bigoplus_\mathbb{N} \pi$ and $\bigoplus_\mathbb{N} \id_{B_0}$ on $\bigoplus_{\mathbb{N}} \cH$ respectively,
we obtain a unital inclusion $(A, \alpha) \subset (B, \beta)$ of the resulting $G$-\Cs-algebras.
As in the statement, both $A$ and $B$ are purely infinite simple nuclear,
and $\alpha$ is amenable.
Furthermore, since the canonical conditional expectation $B \rightarrow B_0$
is $G$-equivariant,
the canonical map $B_0 \rtimes G \rightarrow B \rtimes G$ is injective.
(This follows from Exercise 4.1.4 of \cite{BO} in the discrete group case,
and the same proof works for general locally compact groups.)
Thus $B \rtimes G \neq B \rc G$.
\end{proof}
\subsection*{On amenable models in equivariant Kasparov category}
Recently the second author obtained the first classification result of
\Cs-dynamics beyond amenable groups \cite{Suz20c}.
This motivates to generalize Izumi's conjectures (\cite{IICM}, \cite{IM}, \cite{IM2})
to \emph{amenable actions}.
Recall that, roughly speaking, the original conjecture claims
a bijective correspondence between
(torsion free) \emph{amenable group} actions on Kirchberg algebras (up to appropriate identifications) and certain topological invariants.
The surjectivity part of the original conjecture was recently confirmed by Meyer \cite{Mey}.
Here we prove the amenable action analogue of the surjectivity part for groups with the \emph{Haagerup property} in the stabilized case.
(We refer the reader to the book \cite{CC+} for groups with the Haagerup property.)
This is a consequence of Theorem \ref{Thm:Toeplitz}, a significant theorem of Higson-Kasparov \cite{HK}, and Meyer's results \cite{Mey}.
By (the proofs of) Theorems 3.7 and 3.10 in \cite{Mey},
it suffices to show the following statement.
\begin{Cor}\label{Cor:Hage}
Let $G$ be a second countable group with the Haagerup property.
Then there is an amenable action $\alpha \colon G \curvearrowright \mathcal{O}_\infty \otimes \IK$ with $(\mathcal{O}_\infty\otimes \IK, \alpha) \sim_{\KK^G} \IC$.
\end{Cor}
\begin{proof}
A deep result of Higson--Kasparov
provides a separable nuclear proper $G$-\Cs-algebra $\mathcal{A}(H)$
with $\mathcal{A}(H) \sim_{\KK^G} \mathbb{C}$ (\cite{HK}, Theorems 8.5, 8.6).
As proper $G$-\Cs-algebras are amenable,
Theorem \ref{Thm:Toeplitz}, applied to $\mathcal{A}(H)$,
together with the Kirchberg--Phillips classification theorem \cite{Kir}, \cite{Phi},
gives the desired \Cs-dynamics.
\end{proof}

For free groups, the analogous result holds true in the unital case.
The next theorem is included in the proof of Theorem 5.1 in \cite{Suz19}. By applying Theorem \ref{Thm:Toeplitz}, we obtain a new simpler proof.
\begin{Cor}[\cite{Suz19}]\label{Cor:free}
Let $\Gamma$ be a countable free group.
Then there is an amenable action $\alpha \colon \Gamma \curvearrowright \mathcal{O}_\infty$
with $(\mathcal{O}_\infty, \alpha) \sim_{\KK^\Gamma} \mathbb{C}$.
\end{Cor}
\begin{proof}
We realize $\Gamma$ as a discrete subgroup of $\SL_2(\mathbb{R})$.
Then the left translation action $\Gamma \curvearrowright \SL_2(\mathbb{R})/\SO(2)$
is amenable. As $\SL_2(\mathbb{R})/\SO(2) \cong \mathbb{R}^2$,
we obtain an amenable action
$\alpha \colon \Gamma \curvearrowright \mathbb{R}^2$.
Applying Theorem \ref{Thm:Toeplitz} to $(C_0(\mathbb{R}^2), \alpha)$,
we obtain a $\Gamma$-Kirchberg algebra $(B, \beta)$
with $(B, \beta) \sim_{\mathrm{KK}^\Gamma} (C_0(\mathbb{R}^2), \alpha)$.
In particular $B \sim_{\mathrm{KK}} \mathbb{C}$
by the Bott periodicity.
Moreover the induced action $\Gamma \curvearrowright \mathrm{K}_0(B)$ is trivial.
Take a free basis $S$ of $\Gamma$.
Choose a projection $p \in B$ whose class $[p]_0$ generates $K_0(B)$ \cite{Cun}.
By \cite{Cun}, for each $s\in S$, one can take
$u(s)\in \mathcal{U}(\mathcal{M}(B))$ with $\ad_{u(s)}(\beta_s(p))=p$.
Define a new action
$\zeta \colon \Gamma \curvearrowright B$ to be 
$\zeta_s:= \ad_{u(s)}\circ \beta_s$ for $s\in S$.
Then $p$ is $\zeta$-invariant.
As the unital $\Gamma$-inclusion
$\mathbb{C} \subset pBp$ is KK-equivalent,
by Theorem 8.5 of \cite{MN},
the inclusion is in fact KK$^\Gamma$-equivalent.
By the Kirchberg--Phillips classification theorem \cite{Kir}, \cite{Phi},
we have $pBp \cong \mathcal{O}_\infty$.
\end{proof}
We remark that there are non-exact groups with the Haagerup property \cite{AO}.
Therefore the cocycle perturbation argument used in the above proof
does not work to such groups.

The following observation gives
an obstruction to the existence of an amenable representative
in the category $\mathrm{KK}^G$.
\begin{Prop}\label{Prop:K-isom}
Let $G$ be a locally compact second countable group.
Let $A$ be a separable $G$-\Cs-algebra.
If there is an amenable $G$-\Cs-algebra $B$
satisfying $A\sim_{\mathrm{KK}^G} B$,
then the quotient map $q_A \colon A \rtimes G \rightarrow A \rc G$
is $\KK$-equivalent.
\end{Prop}
\begin{proof}
Let $x \in \mathrm{KK}^G(A, B)$ be a KK$^G$-equivalence.
Then we have the following commutative diagram in ${\rm KK}$:
\[
 \begin{CD}
 A \rtimes G \ @>{j^G(x)}>> B \rtimes G \\
 @V{q_{A}} VV @V{q_{B}}VV \\
 A \rc G @>{j_r^G(x)}>> B \rc G.
 \end{CD}
\]
Here $j^G, j_r^G \colon \KK^G \rightarrow \KK$ denote Kasparov's full and reduced
crossed product functors \cite{Kas} respectively.
Since $x$ is a KK$^G$-equivalence, the row maps are isomorphisms.
As $B$ is amenable, the map $q_B$
is an isomorphism. Thus $q_{A}$ is a $\KK$-equivalence.
\end{proof}
Thus groups satisfying the statement of Corollary \ref{Cor:Hage}
must be K-amenable \cite{CunK}.
\subsection{Exotic construction}
Here we give yet another construction of amenable actions on simple \Cs-algebras.
The construction is an improvement of the one used in the proof of Proposition B in \cite{Suzeq}.

\begin{Thm}\label{Thm:exotic}
Any locally compact group $G$ admits
an action
$\alpha \colon G \curvearrowright A$ on a simple nuclear \Cs-algebra
such that
$A \rtimes_\alpha G = A \rca{\alpha} G$
and that $A \rca{\alpha} G$ is nuclear.
\end{Thm}
In the original article \cite{Suzeq},
we need to assume second countability of $G$.
This is because we have used the Baire category theorem
to take an appropriate amenable action on the Cantor set (cf.~\cite{Suzmin}).
The new construction presented here does not depend on the Baire category theorem.

To construct the desired $G$-\Cs-algebra,
following the strategy in \cite{Suzeq}, we first construct
a suitable amenable action on a compact space.

Let $S$ be an infinite set.
Let $\mathbb{F}_S$ be the free group generated by $S$.
For each $R \subset S$,
let $\Gamma_R$ denote the subgroup of $\mathbb{F}_S$
generated by $R$.
Denote by $\mathfrak{P}_{\rm fin}(S)$ the set of finite subsets of $S$.
\begin{Lem}\label{Lem:af}
There is an amenable free action
$\beta \colon \mathbb{F}_S \curvearrowright \{0, 1\} ^S$.
\end{Lem}
\begin{proof}
For each $F\in \mathfrak{P}_{\rm fin}(S)$,
choose an amenable free action
$\alpha^F \colon \Gamma_F \curvearrowright \{0, 1\}^\mathbb{N}$.
(See e.g., Lemma 2.3 of \cite{Suzmin} for the existence of such an action.)
Let $\beta^F \colon \mathbb{F}_S \curvearrowright \{0, 1\}^\mathbb{N}$
be the action defined to be
\[\beta^F_s := \begin{cases}
 \alpha^F_s & s\in F \\
 \id_{ \{0, 1\}^\mathbb{N}} & s\in S \setminus F.\end{cases}\]
Set \[\beta := \bigotimes_{F \in \mathfrak{P}_{\rm fin}(S)} \beta_F \colon \mathbb{F}_S \curvearrowright (\{0, 1\}^\mathbb{N})^{\mathfrak{P}_{\rm fin}(S)}.\]
By standard calculations of cardinals, $| \mathbb{N} \times\mathfrak{P}_{\rm fin}(S)| = |S|$. Hence 
$(\{0, 1\}^\mathbb{N})^{\mathfrak{P}_{\rm fin}(S)} \cong \{0, 1\} ^S$.
For each $F \in \mathfrak{P}_{\rm fin}(S)$,
the restriction $\beta|_{\Gamma_F}$ factors $\alpha_F$.
Therefore $\beta|_{\Gamma_F}$ is amenable and free.
Hence so is $\beta$.
\end{proof}

\begin{Prop}\label{Prop:minimal}
Let $T \subset S$ be a subset satisfying
$|S|=|T|$.
Then there is an amenable free action
$\alpha \colon \mathbb{F}_S \curvearrowright \{0, 1\} ^S$
whose restriction to $\Gamma_T$ is minimal.
\end{Prop}
\begin{proof}
By Lemma \ref{Lem:af},
one has an amenable free action $\beta \colon \mathbb{F}_S \curvearrowright \{0, 1\} ^S$.
Set $\gamma := \bigotimes_S \beta \colon \mathbb{F}_S \curvearrowright [\{0, 1\} ^S]^ S$.
Note that $X:=[\{0, 1\} ^S]^ S \cong \{0, 1\} ^S$.
Let $\mathrm{CO}(X)$ denote the set of clopen subsets of $X$.
Take a bijection $\varphi \colon T \rightarrow [\mathrm{CO}(X) \setminus \{ \emptyset, X\}]^2$.
Denote by $\varphi(t)=(\varphi(t)_1, \varphi(t)_2)$.
For each $t\in T$,
take a homeomorphism $h_t$ on $X$
such that
\begin{itemize}
\item $h_t(\gamma_t(\varphi(t)_1)) = \varphi(t)_2$,
\item $h_t$ is of the form
\[H \times \id_{\{0, 1\}^{(S \times S) \setminus (W \times F)}}\]
for some $F\in \mathfrak{P}_{\rm fin}(S)$, countable subset $W \subset S$,
and some homeomorphism $H$ on $\{0, 1\} ^ {W \times F}$.
Here we identify $X$ with $\{0, 1\} ^ {W \times F}\times \{0, 1\}^{(S \times S) \setminus (W \times F)}$ in the obvious way.
\end{itemize}
Set $h_s := \id_X$ for $s \in S \setminus T$.
Define $\alpha \colon \mathbb{F}_S \curvearrowright X$
to be $\alpha_s := h_s \circ \gamma_s$ for $s\in S$.
Then, for any proper clopen subsets $Y, Z \subset X$,
the element
$t:= \varphi^{-1}(Y, Z)\in T$ satisfies
$\alpha_t(Y)=Z$.
As $X$ is totally disconnected, this shows that
$\alpha|_{\Gamma_T}$ is minimal.
For each $F \in \mathfrak{P}_{\rm fin}(S)$,
$\alpha|_F$ factors $\beta|_F$.
Hence $\alpha$ is amenable and free.
\end{proof}
\begin{proof}[Proof of Theorem \ref{Thm:exotic}]
Take an infinite set $T$ and a homomorphism
$h \colon \mathbb{F}_T \rightarrow G$ with dense image.
Choose a set $S$ with $T \subset S$, $|S|=|S \setminus T|=|T|$.
Extend $h$ to a homomorphism $k \colon \mathbb{F}_S \rightarrow G$
by setting $k_s =e$ for $s\in S \setminus T$.
By Proposition \ref{Prop:minimal},
one can find an amenable free action $\alpha \colon \mathbb{F}_S \curvearrowright \{ 0, 1\}^S=:X$ such that $\alpha|_{\Gamma_{S \setminus T}}$ is minimal.
Define $\gamma \colon \mathbb{F}_S \curvearrowright X \times G$
to be the diagonal action of $\alpha$ and the left translation action via $k$.
Note that $\gamma$ is amenable, minimal, and free.
Let
$\varphi \colon G \curvearrowright X \times G$ be
the diagonal action of the trivial action and the right translation action.
Then $\gamma$ and $\varphi$ commute.
Therefore $\varphi$ induces a $G$-action
on $C_0(X \times G) \rca{\gamma} \mathbb{F}_S$.
This $G$-\Cs-algebra possesses the desired properties (\cite{AD}, \cite{AS}).
\end{proof}
\subsection*{Acknowledgements}
The authors are grateful to Professors Alcides Buss, Siegfried Echterhoff, and Rufus Willett
for communications on the subject of the article.
We are also grateful to Professor Reiji Tomatsu 
for valuable help concerning the material in Section~\ref{Sec:fixed point}.
This joint work was started during the research camp in Tsurui
village in August 2020. The authors are grateful to the kind hospitality of the people of
Tsurui village. The camp was partly supported by Operator Algebra Supporters' Fund.
Finally we would like to thank the referee for careful reading and helpful comments.

The first author was supported by JSPS KAKENHI (No.~20H01806).
The second author was supported by JSPS KAKENHI Early-Career Scientists
(No.~19K14550) and a start-up fund of Hokkaido University.

\end{document}